\documentclass[11pt]{amsart}
\usepackage[utf8]{inputenc}
\usepackage[a4paper,margin=3cm]{geometry}
\usepackage[pdftex,pdfborder={0 0 0},colorlinks,pagebackref]{hyperref} 
\usepackage{latexsym,lmodern,graphicx,enumitem,amssymb,dsfont}

\let\oldtocsection=\tocsection
 
\let\oldtocsubsection=\tocsubsection
 
\let\oldtocsubsubsection=\tocsubsubsection
 
\renewcommand{\tocsection}[2]{\hspace{0em}\oldtocsection{#1}{#2}}
\renewcommand{\tocsubsection}[2]{\hspace{2em}\oldtocsubsection{#1}{#2}}
\renewcommand{\tocsubsubsection}[2]{\hspace{3em}\oldtocsubsubsection{#1}{#2}}

\newtheorem{thm}{Theorem}
\newtheorem{lem}[thm]{Lemma}
\newtheorem{prop}[thm]{Proposition}
\newtheorem{cor}[thm]{Corollary}
\newtheorem{defi}[thm]{Definition}

\newtheorem{rem}[thm]{Remark}
\newtheorem{ex}[thm]{Example}
\numberwithin{equation}{section}
\numberwithin{thm}{section}

\newcommand{\R}{\mathbb{R}}
\newcommand{\N}{\mathbb{N}}

\newdimen\AAdi%
\newbox\AAbo%
\def\AAk#1#2{\setbox\AAbo=\hbox{#2}\AAdi=\wd\AAbo\kern#1\AAdi{}}%

\renewcommand{\P}{\mathbb{P}}
\newcommand{\E}{\mathbb{E}}

\newcommand{\Var}{\operatorname{Var}}

\def\f{\frac}
\def\Hess{\mathop{\rm Hess}}
\def\<{\langle}
\def\>{\rangle}

\begin{document}
\title[Log-Hessian, deviation bounds 
and regularization effect in $\mathbb{L}^1$]{Log-Hessian and deviation bounds for Markov semi-groups, 
and regularization effect in $\mathbb{L}^1$}
\author{N. Gozlan, Xue-Mei Li, M. Madiman, C. Roberto, P.-M. Samson}

\thanks{Supported by the grants ANR-15-CE40-0020-03 - LSD - Large Stochastic Dynamics; ANR 10-LABX-0058 - Labex Bezout; ANR 11-LBX-0023-01 - Labex 
MME-DII and the grant DMS-1409504 from the U.S. National Science Foundation. Nathael Gozlan, Cyril Roberto and Paul-Marie Samson are supported by a grant of the Simone and Cino Del Duca foundation. This research has 
been conducted within the FP2M federation (CNRS FR 2036).}

\address{Universit\'e de Paris, CNRS, MAP5 UMR 8145, F-75006 Paris, France}
\address{Department of Mathematics, Imperial College London,  London SW7 2AZ, UK}
\address{University of Delaware, Department of Mathematical Sciences, 501 
Ewing Hall, Newark DE 19716, USA.}
\address{Universit\'e Paris Nanterre, Modal'X, FP2M, CNRS FR 2036, 200 avenue de la R\'epublique 92000 Nanterre, France}
\address{LAMA, Univ Gustave Eiffel, UPEM, Univ Paris Est Creteil, CNRS, F-77447 Marne-la-Vall\'ee, France}

\email{nathael.gozlan@u-paris.fr, xue-mei.li@imperial.ac.uk, madiman@udel.edu, croberto@math.cnrs.fr, paul-marie.samson@univ-eiffel.fr}
\keywords{Hypercontractivity, Concentration, Ornstein-Uhlenbeck semi-group, Talagrand's conjecture, Poisson measure, $M/M/\infty$ queue}
\subjclass{60E15, 32F32 and 26D10}

\date{}

\begin{abstract}
{It is well known that some important Markov semi-groups have a ``regularization effect" -- 
as for example the hypercontractivity property of the noise operator on the Boolean hypercube or the 
Ornstein-Uhlenbeck semi-group on the real line, which applies to functions in $L^p$ for $p>1$. 
Talagrand had conjectured in 1989 that the noise operator on the Boolean hypercube has a 
further subtle regularization property for functions that are just integrable, but this conjecture remains open. 
Nonetheless, the Gaussian analogue of this conjecture was proven in recent years by Eldan-Lee and Lehec, 
by combining an inequality for the log-Hessian of the Ornstein-Uhlenbeck semi-group with a new deviation inequality 
for log-semi-convex functions under Gaussian measure.
In this work, we explore the question of how much more general this phenomenon is. 
Specifically, our first goal is to explore the validity of both these ingredients for some diffusion semi-groups in $\mathbb{R}^n$, as well
as for the $M/M/\infty$ queue on the non-negative integers and the Laguerre semi-groups on the positive real line. 
Our second goal is to prove a one-dimensional regularization effect
for these settings, even in those cases where these ingredients are not valid.}
\end{abstract}

\maketitle

\tableofcontents

\section{Introduction}

The aim of this paper is threefold. First, we give explicit formulas for the log-Hessian of some diffusion semi-groups in $\mathbb{R}^n$, 
and explicit lower bounds on some discrete analogue of the log-Hessian for the $M/M/\infty$ queuing process on the non-negative integers $\mathbb{N}:=\{0,1,\dots\}$. 
Second, we investigate deviation bounds for log-semi-convex functions, in 
the above two settings in dimension 1.
Third, we prove in each context an analogue of the Talagrand regularization effect, again in dimension 1, by different means: in the continuous setting of some class of  diffusion semi-groups we generalize the approach developed in \cite{GMRS17,EL18,Leh16} based on the log-Hessian and deviation bounds just mentioned; while for the $M/M/\infty$ queuing process, we 
use a direct computation.

\medskip

We will now present the conjecture by Talagrand, that is the starting point of our investigations, first in its original version on the discrete hypercube and then in the continuous setting of the Ornstein-Uhlenbeck process, before moving to a historical presentation of its resolution in the 
continuous setting and the presentation of our results.

\medskip

Consider the following infinitesimal generator on the $n$-dimensional hypercube
$\Omega_n:=\{-1,1\}^n$, acting on functions as $Lf(\sigma)=\frac{1}{2} \sum_{i=1}^n (f(\sigma^i) - f(\sigma))$. Here $\sigma^i$ is the configuration with the $i$-th coordinate flipped (\textit{i.e.}\ $\sigma_j^i=\sigma_j$ for all $j \neq i$ and $\sigma^i_i=-\sigma_i$). Denote by $(P_s)_{s \geq 0}$ the associated semi-group (sometimes called ``convolution 
by a biased coin" in the literature), and by $\mu_n\equiv 2^{-n}$ the uniform measure on $\Omega_n$ which is reversible for $L$. In \cite{Tal89:1}, Talagrand conjectured (see Conjecture 1 in \cite{Tal89:1}) that for any 
$s>0$, { it holds that }
\[
\lim_{t \to \infty} t \sup_n \sup_{ f \in \mathcal{F}_n} \mu_n(\{ \sigma : P_sf(\sigma) \geq t \}) = 0 { ,}
\]
where
$\mathcal{F}_n:= \{f \colon \Omega_n \to [0,\infty) \mbox{ with } \| f\|_1 = 1 \}$, 
and $\| f\|_p := (\sum_{\sigma \in \Omega_n} |f(\sigma)|^p \mu_n(\sigma))^\frac{1}{p}$ stands for the $\mathbb{L}^p(\mu_n)$-norm of 
$f$, $p \geq 1$. 
Moreover Talagrand formulated the following stronger statement (Conjecture 2 in \cite{Tal89:1}) : 
\begin{equation}\label{eq:Conjecture 2}
t \sup_{ f \in \mathcal{F}_n} \mu_n(\{ \sigma : P_sf(\sigma) \geq t \}) \leq c \frac{1}{\sqrt{\log t}},\qquad\qquad t>1
\end{equation}
for some constant $c=c_s$ depending only on $s$ (and not on $n$). Both conjectures  are still open.

If one assumes that $\| f \|_p =1$ for some $p>1$, then Markov's inequality would give a universal upper bound or order $1/t^{p-1}$ which is much better than $1/\sqrt{\log t}$. 
The hypercontractivity property of the semi-group \cite{Bon70, Bec75} also ensures that, if $f:\Omega_n \to \R$ and $p \geq 1$, then $\|P_sf\|_q \leq \|f\|_p$ with $q = 1+(p-1)e^{2s}$. But this inequality does not say 
anything when $p=1$. Talagrand's conjecture can therefore be seen as a weak $\mathbb{L}^1$ type regularization property of the semi-group. For this reason we may call bounds of the type \eqref{eq:Conjecture 2} a regularization effect in $\mathbb{L}^1$ (or Talagrand regularization effect), even for a fixed dimension $n$.

\medskip

{While} the above problems (Conjectures 1 and 2) are still open, { a recent series of papers deals with a natural continuous counterpart 
to the conjectures, related to the Ornstein-Uhlenbeck semi-group}. Denote 
by $\gamma_n$ the standard Gaussian (probability) measure in dimension $n$, with density 
\[
\mathbb{R}^n \ni x \mapsto (2 \pi)^{-n/2} \exp \left\{ -\frac{|x|^2}{2} \right\} ,
\]
where $|\,\cdot\,|$ denotes the standard Euclidean norm on $\R^n$. For $ p \geq 1$, let $\mathbb{L}^p(\gamma_n)$ be the set of 
measurable functions $f \colon \mathbb{R}^n \to \mathbb{R}$ such that $|f|^p$ is integrable with respect to $\gamma_n$. 
Then, given $g \in \mathbb{L}^1(\gamma_n)$, the Ornstein-Uhlenbeck semi-group is defined by the so-called Mehler representation as
\begin{equation} \label{eq:ou}
P_t^{\text{ou}}g(x) 
:= 
\int g\left(e^{-t}x + \sqrt{1-e^{-2t}}y\right)\,d\gamma_n(y) { ,}
 \qquad \qquad x \in \mathbb{R}^n , \; t \geq 0.
\end{equation}
By a change of variable, we may also write
\begin{equation} \label{eq:ou2}
P_t^{\text{ou}}g(x) = \frac{1}{Z_t}\int g(z) M_t(x,z)dz { ,}
 \qquad \qquad x \in \mathbb{R}^n , \; t \geq 0{ ,}
\end{equation}
where
\[
M_t(x,z)
:= 
\exp\left\{ -\frac{|z-e^{-t}x|^2}{2(1-e^{-2t})} \right\}  
=
e^{ -c_t^2|e^tz-x|^2/2}
\qquad \qquad x,z \in \mathbb{R}^n , \; t \geq 0
\]
and $Z_t = (2\pi(1-e^{-2t}))^{n/2}$ is the normalizing constant and 
\begin{equation}\label{eq:ct}
c_t:=\frac{e^{-t}}{\sqrt{1-e^{-2t}}} \qquad t >0 .
\end{equation}
The semi-group $(P_t^{\text{ou}})_{t \geq 0}$ is associated to the infinitesimal diffusion operator 
$L^{\text{ou}}:=\Delta - x \cdot \nabla$ and enjoys the exact same hypercontractivity property as the convolution by biased coin operator 
on the discrete hypercube defined above. It is therefore natural to ask for an upper bound for 
\[
S_s(t):=t \sup_{ f \geq 0, \| f\|_1 =1} \gamma_n(\{ \sigma : P_s^{\text{ou}}f(\sigma) \geq t \}), \qquad s>0 .
\]
In \cite{EL18,Leh16} Eldan, Lee and Lehec fully solved the problem by proving that
for any $s >0$ there exists a constant $c_s \in (0, \infty)$ (depending only on $s$ and not on the dimension $n$) such that $S_s(t) \leq \frac{c_s}{\sqrt{\log t}}$ for all $t>1$
and this bound is optimal in the sense that the factor $\sqrt{\log t}$ cannot be improved.
In an earlier paper \cite{BBBOW13}, Ball, Barthe, Bednorz, Oleszkiewicz and Wolff already obtained a similar bound but with a 
constant $c_{s,n}$ depending on the dimension $n$ plus some extra $\log \log t$ factor in the numerator. 
Later Eldan and Lee \cite{EL18}, using tools from stochastic calculus, proved that the above bound holds 
with a constant $c_s$ independent on $n$ but again with the extra $\log \log t$ factor in the numerator. 
Finally Lehec \cite{Leh16}, following \cite{EL18}, removed the $\log \log 
t$ factor.

\medskip

In both Eldan-Lee and Lehec's papers, the two key ingredients are the following: 
\begin{itemize}
\item[(1)]  \textbf{log-semi-convexity:}\\
for any $s>0$, the Ornstein-Uhlenbeck semi-group satisfies, for all non-negative function $g \in \mathbb{L}^1(\gamma_n)$,
\[
\mathrm{Hess}\,(\log P_s^{\text{ou}} g) \geq -c_s^2 \mathrm{Id},
\]
where $\mathrm{Hess}$ denotes the Hessian matrix and $\mathrm{Id}$ the identity matrix of $\mathbb{R}^n$.\\
\item[(2)]  \textbf{deviation for log-semi-convex functions:}\\for any positive function $g$ with $\mathrm{Hess}\,(\log g) \geq -\beta \mathrm{Id}$, for some $\beta \geq 0$, and satisfying $\int g \,d\gamma_n =1$, one 
has
\[
\gamma_n(\{g \geq t \}) \leq \frac{C_\beta}{t \sqrt{\log t}} \qquad \qquad \forall t>1,
\]
with $C_\beta = \alpha \max(1, \beta)$.
\end{itemize}

\medskip

Let us underline that the difficulty of the questions raised by Talagrand 
completely relies on the uniformity in the dimension $n$. For simplicity we deal with the discrete setting but 
similar considerations could be done in the continuous as well (see \cite{BBBOW13}).
For a fixed integer $n$, proving \eqref{eq:Conjecture 2} with a constant $c$ depending on $s$ and on the dimension $n$ is easy.

This can be seen using for instance the following line of reasoning. 
 Observe that for all $f :\Omega_n \to \R$, it holds
\[
P_sf(\sigma)=\int f(\eta) K_s(\sigma,\eta)d\mu_n(\eta),
\]
with $K_s(\sigma,\eta)=\prod_{i=1}^n (1+e^{-s}\sigma_i\eta_i)$ and so
\[
\sup_{f \in \mathcal{F}_n} P_sf(\sigma) = \sup_{\eta \in \Omega_n} K_s(\sigma,\eta) = (1+e^{-s})^n ,\qquad \qquad \forall \sigma \in \Omega_n
\]
Therefore, for $t \geq0$,
\begin{align*}
t \sup_{ f \in \mathcal{F}_n} \mu_n(\{ \sigma : P_sf(\sigma) \geq t \})
& \leq 
t  \mu_n(\{ \sigma : \sup_{ f \in \mathcal{F}_n} P_sf(\sigma) \geq t \})  
= t \mathbf{1}_{\{t\leq (1+e^{-s})^n\}}.
\end{align*}
In particular,  
\[
t \sup_{f\in \mathcal{F}_n}\mu_n(\{ \sigma : P_sf(\sigma) \geq t \}) = 0
\] 
as soon as $t>(1+e^{-s})^n$ and so, for any fixed $s>0$ and $n\in \N$ it clearly exists a constant $c = c_{s,n}$ such that \eqref{eq:Conjecture 2} is satisfied. 

We may call this approach the ``strategy of the uniform bound on $P_t$" in the reminder of the paper. 
As observed in \cite{BBBOW13,GMRS17}, there does exist an other strategy based on uniform bounds, but with a different flavor. 
To illustrate this, in the continuous now, we recall a result from  \cite{GMRS17} (that we generalize in dimension 1 in Lemma \ref{lem:technical}): any
$g \colon \mathbb{R}^n \to (0,\infty)$ with $\int g d\gamma_n=1$,  smooth and such that $\mathrm{Hess}\,(\log g) \geq -\beta \mathrm{Id}$ for some $\beta \geq 0$,
satisfies
\begin{equation} \label{eq:g}
g(x) \leq (1+\beta)^\frac{n}{2} e^{\frac{1}{2}|x|^2} , \qquad \forall x \in \mathbb{R}^n .
\end{equation}
Therefore, if one knows \textit{a priori} that $P_tf$ is log-semi-convex, 
uniformly in $f$, then $P_tf(x) \leq (1+\beta)^\frac{n}{2} e^{\frac{1}{2}|x|^2}$ with some $\beta$ depending only on $t$
and deviation bounds for $P_t f$ would follow from deviation bounds for $(1+\beta)^\frac{n}{2} e^{\frac{1}{2}|x|^2}$ (a quantity which does not depend on $f$ anymore).
To distinguish the two approaches, we may call the latter the "strategy of the uniform bound for log-semi-convex functions". 

As pointed out in \cite{BBBOW13,GMRS17} and as one can realize from the above discussion, the uniform bounds depend on the dimension $n$, therefore there is no hope to prove Talagrand's conjectures by means of any of the above two strategies.

\medskip

At this point, we note that Problem (1) (log-semi-convexity) is closely connected to certain areas of geometric analysis.
On the one hand, it is at the heart of some of the fundamental problems in the Analysis of Loop Spaces. A program of Gross \cite{Gro75} is to prove
Logarithmic Sobolev and Poincar\'e inequalities from Gaussian measures to 
Brownian motion and conditioned Brownian motion measures.
The main problem involves constructing an Ornstein-Uhlenbeck process on the space of loops,  
obtaining integration by parts formula for these measures,  and Poincar\'e inequalities. 
The latter is notoriously difficult, with counter examples by Eberle \cite{Ebe02} and defective inequalities by Gong-Ma \cite{GM98}. 
The Poincar\'e inequality is only proven to hold for  very few classes of 
manifolds: see Aida \cite{Ai00} 
for asymptotically flat manifolds and Chen-Li-Wu \cite{CLW} for hyperbolic spaces. 
The idea is to compare $\log p_t$, its gradient and Hessian with 
that of the Heat kernel on the Euclidean space and 
one wishes to obtain information on $ t\mathrm{Hess}\log p(t,x,y)+\mathrm{Hess}\left(\frac{d^2(x,y)}{2}\right)$.
Problem (1) is also closely related to the Li-Yau inequality and the extensive literature it has generated
in Geometric Analysis \cite{Li-Yau}.


\medskip

From the above considerations it seems natural on the one hand to investigate on log-semi-convexity and deviation bounds for log-semi-convex functions (Problems (1) and (2)) that are of independent interest, and on the other hand to explore their connections with Talagrand's regularization 
effect in $\mathbb{L}^1$. Before entering into a detailed description of the content of the paper, we can already quote that we will give non-trivial results essentially concerning log-semi-convexity (Problem (1)); obtaining dimension free estimates for Problem (2) remains an open (and we believe interesting) question that we do not address here. Our results also 
indicate that the approach by Eldan-Lee and Lehec (consisting on combining (1) and (2)) to prove the Talagrand Conjecture may fail in other settings, which, we believe, has it own merit. Besides, this emphasizes the fact that finding an alternative proof of the conjecture, even in the setting of the Gaussian measure and the Ornstein-Uhlenbeck process, would be of 
interest for possible generalizations.

\medskip

The content of the paper goes as follows.

\smallskip

In Section~\ref{sec:continuous}, we  investigate log-semi-convexity for general diffusion semi-groups $(P_s)_{s\geq0}$, in any dimension.
In fact, using the Feynman-Kac formula, we are able to give an explicit representation for $\mathrm{Hess}(\log P_s g)$ that leads, 
under some assumptions, to a similar bound {as} in the log-semi-convexity 
property of Problem $(1)$.
We also investigate  deviation for  log-semi-convex functions for diffusions, in dimension 1 only, of the form $L=\frac{d^2}{dx^2} - h'\frac{d}{dx}$ 
(which corresponds to the Ornstein-Uhlenbeck operator for the choice $h(x)=\frac{1}{2}x^2$),
when $0< c \leq h'' \leq C$. Our results might therefore  be seen as perturbations (though potentially unbounded) of the Ornstein-Uhlenbeck setting.
Then we apply the approach developed in \cite{GMRS17} to prove Talagrand's regularization effect for such diffusions in dimension 1.

\smallskip

In {Section~\ref{sec:mmi}}, we investigate log-semi-convexity and deviation for log-semi-convex functions (\textit{i.e.}\ Problem $(1)$ and $(2)$) 
in the discrete setting of the $M/M/\infty$ queuing process on the integers. 
We will prove that, in that setting, a result similar to that of Problem $(1)$ still holds. On the other hand it appears that the picture is 
very different for Problem $(2)$ { in the discrete setting}. In fact, if $g$ is { ``log-convex''}, in the sense that 
$\Delta \log g \geq 0$, where $\Delta$ is a discrete analogue of the Laplacian (see Section \ref{sec:mmi} for the definition), 
then a deviation bound similar to Problem (2) holds. In contrast, we will 
construct counterexamples of the result of Problem $(2)$ for $g$ 
satisfying $\Delta \log g = - \beta$, with $\beta >0$. The first property {transfers} to Talagrand's regularization effect for the $M/M/\infty$ 
queuing process. More precisely, if $g$ is ``log-convex'' (in the discrete sense), then the strategy developed in  \cite{GMRS17} (strategy of the uniform bound for log-convex functions) leads to a 
positive conclusion regarding the regularization effect but \emph{restricted} to convex functions.
However,  as shown in Section \ref{sec:tal-M/M/infty}, the strategy of the uniform bound on $P_t$ presented above appears to be powerful in the 
case of the $M/M/\infty$ semi-group on the integers and will allow us to (fully) prove the regularization effect in this setting.

\smallskip

It should be noticed here that the strategy of the  uniform bound on $P_t$ holds in the case of the Ornstein-Uhlenbeck semi-group in dimension 1 \cite{BBBOW13}, 
but does not seem to apply to the perturbations of the Ornstein-Uhlenbeck 
considered in this paper.
Therefore the situation is very different between the continuous and the discrete setting, and somehow in opposition
 (at least for the $M/M/\infty$ queuing process and the family of diffusion semi-groups we consider): 
 the strategy of the uniform bound on $P_t$ works in the discrete, but not in the continuous; 
 in contrast, the strategy consisting of combining log-semi-convexity and 
deviation bounds for log-semi-convex functions (\textit{i.e.}\ $(1)$ and $(2)$ above) works in the continuous, but fails in the discrete setting. 
 As pointed out by the referee this could be a consequence of our choice of the discrete operator $\Delta$. In fact, in discrete settings, there often exist
different natural definitions of the objects under consideration, each of 
them with its own interests and advantages. It could be that, for a different choice of generator, log-semi-convexity would imply some regularization effect. It would worth exploring such a direction. 

\smallskip

In fact, as will be shown in Section~\ref{sec:Laguerre} by considering yet another class of semi-groups
(namely, the Laguerre semi-groups on $(0,\infty)$), the picture can be different from the two previous ones. Namely we will show 
that neither the  log-semi-convexity property, nor the deviation for log-semi-convexity functions property holds for the Laguerre semi-group, but the analogue of Talagrand's regularization effect in $\mathbb{L}^1$ still 
holds.

\smallskip

We may summarize the different situations in the following diagram (in the present paper we investigate and prove results in the last three columns):
\par\vspace{.2in}

\begin{tabular}{|c|c|c|c|c|}
\hline
Semi-group: & Ornst.-Uhl. & $0 < c \leq h'' \leq C$ & $M/M/\infty$ & Laguerre \\
\hline
Pb (1): Lower  &  & Yes (under some  &  &  \\
bound on &  Yes & assumptions & Yes &  No     \\
  $(\log P_tf)''$ &   & on $h$) & & \\
\hline
Pb (2): Deviation  &  & Yes (under some &  &  \\
bounds for semi- & Yes & assumptions &   No ($\beta >0$)  & No ($\beta>0$) \\
log-convex functions & & on $h$)  & Yes ($\beta=0$)   &   			\\
 ($(\log f)''\geq -\beta$) & & &  & \\
\hline
Regularization & (1) + (2)  &  (1) + (2) & unif. bound  & unif. bound\\
effect in $\mathbb{L}^1$  & \cite{EL18,Leh16,GMRS17} or & $\frac{1}{t \sqrt{\log t}}$ & $\frac{\sqrt{\log \log t}}{t \sqrt{\log t}}$ & $\frac{1}{t 
\sqrt{\log t}}$
 \\
dim $n=1$  & unif. bound \cite{BBBOW13} & & & \\
\hline 
Talagrand's & (1) + (2)  &  & & \\
conjecture & \cite{EL18,Leh16}   & unknown& unknown & unknown\\
 dim $n>1$ & & &  & \\
\hline 
\end{tabular}


\bigskip

Once more, we emphasize that although the results of this paper in regards to the validity of Talagrand's conjecture are limited  
to certain semi-groups in one dimension, one may hope that our exposure of the variety of situations that may occur
illustrate to some extent both the potential robustness of the underlying 
phenomenon of smoothening of integrable functions, 
as well as the non-robustness of proof techniques to demonstrate the same. We also note that the smoothening
effect of Markov semi-groups, together with their ability (when ergodic) to interpolate between arbitrary starting points and the 
invariant measure, has been used an innumerable number of times as a proof tool, such as in the proof of 
functional or entropy inequalities (see, e.g., \cite{Bar86, ABBN04:1, MB06:isit, Bor03, Joh07, BCCT08, JKM08:maxent, BH09});
one may hope that the results of this paper are of some use in such investigations.

\section{Diffusion semi-groups} 
\label{sec:continuous}

In this section, we derive an explicit formula for the Hessian (second-order space derivative) of $\log P_t f$ for general diffusion semi-groups
 which is the technical heart of our results for such semi-groups. The formula in given in Theorem \ref{thm:loghessian} below. Its proof being rather technical it is postponed to Section~\ref{ss:log-hess}.

For the reader convenience, we may start with a concrete application of such a formula and its consequence in terms of regularization effect in $\mathbb{L}^1$
for some class of diffusion, in dimension 1.

In order to present our results, we need some notation:



\begin{itemize}
\item $\mathcal{C}_K^\infty$ denotes the set of $\mathcal{C}^\infty$ real 
valued functions with compact support.
\item $\mathcal{D}^{(n)}$, $n =0,1,...$, denotes the set of $\mathcal{C}^n$ real valued functions whose derivatives and the function itself have 
polynomial growth.
\item The outer product of two vectors of $\mathbb{R}^n$, $x=(x_1,\dots,x_n)$, $y=(y_1,\dots,y_n)$
is defined as usual as $u \otimes v := uv^T$, \textit{i.e.}\ $u \otimes 
v$ is the $n \times n$ matrix with entries $(u \otimes v)_{ij}=u_iv_j$, 
$1 \leq i,j \leq n$. 
\end{itemize}

The log-Hessian formula of Theorem \ref{thm:loghessian} below leads to the following useful and tractable result (related to log-semi-convexity (Problem (1)) in the introduction) in dimension 1. Its proof is clear from Theorem \ref{thm:loghessian} and the fact that
$\int_0^t \left(\frac{\sinh(t-s)}{\sinh(t)} \right)^2 ds = \frac{e^{2t}- e^{-2t} - 4t}{2(e^{2t}+e^{-2t}-2)} \leq \frac{1}{2}$ for $t \geq 0$.

\begin{prop}[log-semi-convexity] \label{prop:loghessian}
Let $h \colon \mathbb{R} \to \mathbb{R}$ be of class $\mathcal{D}^{(4)}$ with  $\mu_h(dx):=e^{-h(x)}dx$  a finite probability measure. 
Set  $V(x):= \frac{1}{2}(1-h''(x)) - \frac{1}{4}(x^2-{h'}^2(x))$, $x \in \mathbb{R}$ and assume that
$V$ is bounded from below. 
to the diffusion operator $L_h:=\frac{\partial^2}{\partial x^2} - h'(x) 
\frac{\partial}{\partial x}$.
Then, given $f \in \mathbb{L}^2(\mu_h)$ non-negative, for all $x \in \mathbb{R}$ and $t \geq 0$,  one has
\begin{equation} \label{prop:froid}
(\log \mathcal{P}_t f)'' (x) \geq  -c_t^2 - \frac{1}{2}(1-h''(x)) - \frac{1}{2} \sup_{y \in \mathbb{R}} V''(y) 
\end{equation}
where  $(\mathcal{P}_t)_{t \geq 0}$ is the semi-group associated to the diffusion operator $L_hf:= f'' - h'f'$.
\end{prop}
 
As an example of application, one can consider $h(x)=\frac{x^2}{2} + (1+x^2)^{p/2}$, with $p \leq 2$. Then it is easy to see that $h$ satisfies the hypotheses of the latter and that $- \frac{1}{2}(1-h''(x)) - \frac{1}{2} \sup_{y \in \mathbb{R}} V''(y) \geq -c_p$ for some constant $c_p$ depending only on $p$.

\medskip

As already mentioned, the lower bound on $(\log \mathcal{P}_t f)''$ of Proposition \ref{prop:loghessian} is a consequence of a much more general result (exact formula, any dimension) that we now present. The idea behind 
its proof is to use two perturbation arguments. The first one is based on 
the so-called Feynman-Kac formula that allows one to represent the semi-group of the (perturbed) operator $L^V=L - V$ ($V$ acting multiplicatively) in term of the process associated to $L$ (in our case, $L=L^{ou}$ is the Ornstein-Uhlenbeck operator). This leads to an explicit representation for the Hessian of $\log P_t^V$ (Theorem \ref{loghessian}).
Then, by means of a $h$-transform, one can perturb $L^V$ again to reach the desired diffusion $\Delta - \nabla h \cdot \nabla$ and the following theorem which is one of our main results (a more complete version can be found in Corollary \ref{cor:loghessian}).

\begin{thm} \label{thm:loghessian}
Let $h \colon \mathbb{R}^n \to \mathbb{R}$ belongs to $\mathcal{D}^{(4)}$. Set $\mu_h(dx)=e^{-h(x)}dx$. Denote by $(\mathcal{P}_t)_{t \geq 0}$ 
the semi-group associated to the diffusion operator $\Delta - \nabla h \cdot \nabla$.
Put
$$
W(x):=\frac{|x|^2}{2} - h(x), \quad \qquad 
V(x):= \frac{1}{2}(n-\Delta h) - \frac{1}{4}(|x|^2-|\nabla h|^2), \quad 
x \in \mathbb{R}^n.
$$ 
Assume that $V$ is bounded from below. Let $f \in \mathbb{L}^2(\mu_h) \setminus \{0\}$ be non negative and, for all $x \in \R^n$, denote by $\E_{f,x}^W$ the expectation with respect to  the probability measure $\mathbb{Q}_{e^{\frac W 2} f, x}$ introduced in Definition \ref{def:Q}.
at $x$. 
Then
\begin{align}\label{thm:froid}
\mathrm{Hess}(\log \mathcal{P}_t f)(x) 
& = 
\quad - \frac{1}{2}(\mathrm{Id}-\mathrm{Hess}(h)(x))  + \mathbb{E}_{f,x}^W (A_t^x\otimes A_t^x) 
- \mathbb{E}_{f,x}^W (A_t^x)\otimes  \mathbb{E}_{f,x}^W (A_t^x)  \nonumber \\
& \quad  -c_t^2\mathrm{Id} - \int_0^t \left( \frac{\sinh (t-s)}{\sinh(t)}\right)^2 \mathbb{E}_{f,x}^W (  \Hess V(X_s^x)) ds
\end{align}  
where
\[
A_t^x:=- \int_0^t \nabla V(X_s^x)\frac{\sinh (t-s)}{\sinh(t)}\, ds 
+\f{e^{-t}} {1-e^{-2t}} (X_t^x-e^{-t}x)
\]
and $c_t$ is given by \eqref{eq:ct}.
\end{thm}
 
The proof is given in the next section (see the proof of Corollary \ref{cor:loghessian}).

Heuristically, one can see that $W$ is devised to transform the Gaussian measure into the measure $e^{-h}$. On the other hand, in 
\eqref{thm:froid}, the first term $\mathrm{Id}-\mathrm{Hess}(h)(x)$ measures the discrepancy of $h$ with respect to $|x|^2/2$ while  the second term is some sort of (co)variance term, that one can hope to be positive. Obviously, $h= \frac{|x|^2}{2}$ leads to $W=V=0$ and the Ornstein-Uhlenbeck semi-group.

\medskip

Proposition \ref{prop:loghessian} above is concerned with log-semi-convexity. In the next section we derive, in dimension 1, some deviation bounds 
for log-semi-convex functions and in turn prove a regularization effect in $\mathbb{L}^1$ for a class of diffusion. 
Finally, in Section \ref{ss:log-hess} we prove Theorem \ref{thm:loghessian}.

%
%
%
%

\subsection{Deviation bounds for log-semi-convex functions}
\label{ss:diff-slc}

The aim of this section is to prove a deviation bound for log-semi-convex 
functions, following \cite{GMRS17}. Namely, the following holds.

\begin{thm} \label{bis}
Let $\mu_h$ be a probability measure on $\R$ of the form $d\mu_h(x) = e^{-h(x)}\,dx$ with $h \colon \mathbb{R} \to \mathbb{R}$ a symmetric $\mathcal{C}^2$ function. Assume that there exist $c,C >0$ such that $c \leq  h'' \leq C$.
Then, for any  $\mathcal{C}^2$ function $f \colon \mathbb{R} \to (0,\infty)$  such that $(\log f)''\geq - \beta$ for some $\beta \geq 0$, it holds
\[
\mu_h \Big(\left \{ f \geq t \int f\,d\mu_h \right\}\Big) \leq  \Big(\frac{C+\beta}{c}\Big) \; \frac{1}{t \sqrt{\log t}} , \qquad \qquad \forall t 
\geq 2 .
\]
\end{thm}

\begin{rem}
The assumption $h$ symmetric is here for simplicity. A similar statement would hold with $h$ non symmetric. The special case $h(x)=x^2/2$ is given in \cite{GMRS17} with a factor $(1+\beta)/\sqrt{2}$ which is slightly better than $1+\beta$ (since $c=C=1$ when $h(x)=x^2/2$).
\end{rem}

The proof of Theorem \ref{bis} relies on the following technical lemma whose proof can be found at the end of this section.

\begin{lem} \label{lem:technical}
Let $\mu_h$ be a probability measure on $\R$ of the form $d\mu_h(x) = e^{-h(x)}\,dx$ with $h \colon \mathbb{R} \to \mathbb{R}$ a symmetric $\mathcal{C}^2$  function.
Assume that there exists $C >0$ such that $0\leq  h'' \leq C$.
Then, for any $\varphi \colon \mathbb{R} \to \mathbb{R}$ of class $\mathcal{C}^2$ such that $\varphi''\geq - \beta$ for some $\beta \geq 0$, it holds
\[
\varphi(x) - \log \left( \int e^\varphi\,d\mu_h\right) \leq \frac{1}{2}\log\left(\frac{C+\beta}{2\pi}\right) + h(x),\qquad \qquad \forall x \in \R.
\]
\end{lem}

\begin{proof}[Proof of Theorem \ref{bis}]
Set $\varphi = \log f$, which satisfies $\varphi'' \geq -\beta$. Without loss of generality one can assume that $\int e^\varphi\,d\mu_h =1$. Define $a = \frac{1}{2}\log \left(\frac{C+\beta}{2\pi}\right)$. From Lemma \ref{lem:technical} and by symmetry of $h$ we have, for all $t > 2(a+h(0))$
\[
\mu_h(\{\varphi \geq t \})  \leq 
\mu_h( \{ h(x) \geq t - a \}) 
\leq 
2 \int_{h^{-1}(t - a)}^\infty e^{-h(x)}dx  \leq  
\frac{2 e^ae^{-t}}{h'(h^{-1}(t - a)))}
\]
where we used the following bound, valid for any $s>0$ (recall that $h'$ is increasing on $\mathbb{R}^+$)
\[
\int_s^\infty e^{-h(x)} dx \leq \int_s^\infty \frac{h'(x)}{h'(s)}e^{-h(x)} dx = \frac{e^{-h(s)}}{h'(s)} . 
\]
Now observe that, since $h$ is smooth and symmetric, $h'(0)=0$ so that
$h(x) \leq h(0) + \frac{1}{2}Cx^2$ and $h'(x) \geq cx$, $x \geq 0$.
Therefore
\[
h'(h^{-1}(x)) \geq h'\left(\sqrt{\frac{2(x-h(0))}{C}}\right) \geq c\sqrt{\frac{2(x-h(0))}{C}} \qquad \mbox{for any } x \geq h(0).
\]
In turn, since we fixed $t \geq 2(a+h(0))$, $2\left((t - a) - h(0)\right)\geq t$ and
thus, thanks to the latter
\[
h'(h^{-1}(t -a))) \geq c\sqrt{\frac{2\left((t - a) - h(0)\right)}{C}} \geq c\sqrt{\frac{t}{C}}.
\] 
We  conclude that, for any $t \geq2(a+h(0))$, 
\[
\mu_h(\{\varphi \geq t \}) \leq  \frac{2\sqrt{C}e^a}{c} \frac{e^{-t}}{\sqrt{t}}
\leq 2 \frac{C+\beta}{c\sqrt{2\pi}} \frac{e^{-t}}{\sqrt{t}} .
\]
Next we deal with $ t \in (0, 2(a+h(0)))$.  Using Markov's inequality, since $\int e^\varphi d\mu_h=1$, we have
\[
\mu_h(\{\varphi \geq t \}) \leq e^{-t} \leq \sqrt{2(a+h(0))} \frac{e^{-t}}{\sqrt{t}}
 .
\]
Since $\int e^{-h}\,dx=1$ and $h(0) + c\frac{x^2}{2} \leq h(x)$, we have 
$2h(0) \leq \log \frac{2 \pi}{c}$ so that
\begin{align*}
\sqrt{2a+2h(0)}
& \leq 
\sqrt{\log ((C+\beta)/c)} \leq \frac{C+\beta}{c}
\end{align*}
where the last inequality follows from a direct computation. 
\end{proof}

\begin{proof}[Proof of Lemma \ref{lem:technical}]
We follow \cite[Lemma 2.1]{GMRS17}. The bound is trivial if $\int e^\varphi\,d\mu_h = +\infty$ so let us assume that $\int e^\varphi d\mu_h=1$.
Define $g(x) = \varphi(x) - h(x) + \alpha\frac{x^2}{2}$, $x\in \R$, with $\alpha = C+\beta$. The function $g$ is convex on $\R$ and so, by Fenchel-Legendre duality, it holds
$g(x)= \sup_{y \in \mathbb{R}} \left\{ xy - g^*(y) \right\}$, $x\in \R$, where 
$g^*(y):= \sup_{x \in \mathbb{R}} \left\{ yx - g(x) \right\}$, $y\in \R$, is the convex conjugate of $g$. 
Therefore, for all $y \in \R$, 
\[
1 = \int e^{\varphi(x)-h(x)}\,dx = \int e^{g(x)- \alpha \frac{x^2}{2}}\,dx \geq e^{-g^*(y)} \int e^{xy- \alpha \frac{x^2}{2}}\,dx = e^{-g^*(y)} \sqrt{\frac{2\pi}{\alpha}}e^{\frac{y^2}{2\alpha}}.
\]
So $g^*(y) \geq \frac{1}{2} \log \left(\frac{2\pi}{\alpha}\right) + \frac{y^2}{2\alpha}$, for all $y\in \R.$ Therefore,
\[
g(x) \leq \frac{1}{2} \log \left(\frac{\alpha}{2\pi}\right) + \sup_{y\in \R} \left\{xy-\frac{y^2}{2\alpha}\right\} =  \frac{1}{2} \log \left(\frac{\alpha}{2\pi}\right)  + \alpha\frac{x^2}{2},
\]
which proves the claim.
\end{proof}

.

\subsection{Regularization effect in $\mathbb{L}^1$ for a class of diffusion, in dimension 1}
\label{ss:diff-tal}

In this section we prove that for some class of potentials $h$, the associated diffusion semi-group satisfies the Talagrand Regularization effect, 
in dimension 1. 
Theorem \ref{semifinal} below is a corollary of the results of the previous section and Proposition \ref{prop:loghessian}.

\begin{thm} \label{semifinal}
Let $\mu_h$ be a probability measure on $\R$ of the form $d\mu_h(x) = e^{-h(x)}\,dx$ with $h \colon \mathbb{R} \to \mathbb{R}$ a symmetric function of class $\mathcal{D}^{(4)}$ such that $c \leq  h'' \leq C$ where $c, 
C$ are positive numbers. 
Set $V(x):= \frac{1}{2}(1-h'') - \frac{1}{4}(x^2-{h'}^2)$. Assume $V$   
is bounded below, with $\sup_{x \geq 0}V''(x) < \infty$. Finally denote by $(\mathcal{P}_t)_{t \geq 0}$ the semi-group associated to the diffusion 
operator $L_h:=\frac{\partial^2}{\partial x^2} - h'(x) \frac{\partial}{\partial x}$ symmetric in $\mathbb{L}^2(\mu_h)$. 

Then, for all $s>0$, there exists a constant $D$ (that depends only on $s$, $c$, $C$ and $\sup_{x \geq 0}V''(x)$) 
such that for all non-negative $g  \in \mathbb{L}^1(\mu_h)$
$$
\mu_h ( \{\mathcal{P}_s g \geq t \int g d\mu_h\} ) \leq  D \frac{1}{t\sqrt{\log t}} \qquad \qquad \forall t \geq 2 .
$$
\end{thm}

\begin{ex}
As an example of application, one can consider $h(x)=\frac{x^2}{2} + (1+x^2)^{p/2}$, with $p \leq 2$ which satisfies the assumption of the Theorem. Note that this example corresponds to an unbounded perturbation of the Gaussian potential.

Many bounded perturbations of the Gaussian potential also enter the framework of the above theorem. However, due to the assumption
$V$ bounded below, even apparently very tiny perturbation of the Gaussian 
potential does not enter the framework of the theorem, as for example $h(x)=\frac{x^2}{2} + \cos (x)$! We believe that the reason is technical and that both the regularization effect and the Talagrand's conjecture should hold also in this case.
\end{ex}

\begin{proof}
Fix $g  \in \mathbb{L}^2(\mu_h)$ positive and $s >0$.
Thanks to Proposition \ref{prop:loghessian}, 
\[
(\log \mathcal{P}_s g)'' \geq -c_s^2 -\frac{1}{2}(1 -c) - \frac{1}{2} \|V''\|_\infty \geq - \beta 
\]
with $\beta := \max(0,c_s^2 +\frac{1}{2}(1 -c) + \frac{1}{2} \|V''\|_\infty) \geq 0$.
Therefore, by Theorem \ref{bis} applied to $f=\mathcal{P}_s g$, one can 
conclude that,
for all $t\geq 2$,
\[
\mu_h ( \{ \mathcal{P}_s g \geq t \int g\,d\mu_h\})  \leq \frac{C+\beta}{c} \frac{1}{t\sqrt{\log t}}
\]
which is the desired conclusion for $g \in \mathbb{L}^2(\mu_h)$. Applying 
the previous bound to $g\wedge n$, $n\geq 1$, for non-negative $g\in\mathbb{L}^1(\mu_h)$ and letting $n\to \infty$ completes the proof.
\end{proof}

\subsection{Warm up: bounds on the Ornstein-Uhlenbeck semi-group in dimension 1} 
\label{ss:OU}

In this section we deal with the dimension 1 for simplicity, and set $\gamma:=\gamma_1$ with density 
$\varphi(x)=(2 \pi)^{-1/2}e^{-x^2/2}$, $x \in \mathbb{R}$. Put
\[
H_n(x):=e^{x^2/2} (-1)^n \frac{d^n}{dx^n} \left( e^{-x^2/2}\right) 
\]
for the Hermite polynomial of degree $n=0,1,\dots$, with the convention 
that $H_0 \equiv 1$. It is well known that the family of Hermite polynomial is an orthonormal basis of $\mathbb{L}^2(\gamma)$. Simple computations 
lead to
$H_1(x)  =x$, $H_2(x)  =x^2-1$, $H_3(x)  = x^3-3x$ etc.
Now, by a direct induction argument, the following identities hold:
\[
(P_t^{\text{ou}}g)^{(n)}(x) = c_t^n  \int g\left(e^{-t}x + \sqrt{1-e^{-2t}} y\right) H_n(y) d\gamma(y),\qquad\qquad n\in \N,
\]
with $c_t$ defined by \eqref{eq:ct}.
Fix a positive integrable function $g$ and, for any $x \in \R$, denote by 
$\mathbb{E}_x$ the expectation with respect to the probability measure 
with density 
\[
y \mapsto g\left(e^{-t}x + \sqrt{1-e^{-2t}}y\right)  /  \int g\left(e^{-t}x + \sqrt{1-e^{-2t}}y\right) d\gamma(y)
\]
with respect to the Gaussian measure $\gamma$.
The above identities then read
\[
d_n(x):=\frac{(P_t^{\text{ou}}g)^{(n)} (x)}{P_t^{\text{ou}}g(x)} = c_t^n \mathbb{E}_x(H_n(Y)),\qquad \qquad x \in \R, n \in  \N.
\]

Our next step is to explore the first derivatives of $x \mapsto u_t(x):=\log P_t^{\text{ou}}g (x)$. Letting for simplicity $g_t(x):=P_t^{\text{ou}}g(x)$, 
we get after simple algebra 
\begin{align*}
u_t' (x)&= \frac{g_t'}{g_t}(x) = d_1(x) = c_t \mathbb{E}_x[H_1(Y)] = c_t \E_x[Y] \\
u_t''(x) &= \frac{g_t''}{g_t}(x) - \left(\frac{g_t'}{g_t}\right)^2(x) = 
d_2(x)-d_1^2(x) = c_t^2 \left( \mathbb{E}_x[H_2(Y)] - \mathbb{E}_x[H_1(Y)]^2 \right) \\
& = 
c_t^2 \left(-1+ \mu_2 (x)\right) ,
\end{align*}
where $\mu_2(x) = \E_x[Y^2] - \E_{x}[Y]^2 \geq0$.
In particular, 
\[
(\log P_t^{\text{ou}}g)'' (x) = u_t''(x) \geq -c_t^2 
\]
which corresponds to the log-semi-convexity property (Problem (1)) in the 
Introduction.

\subsection{Representation for the Hessian of perturbed Ornstein-Uhlenbeck semi-groups} 
\label{ss:log-hess}
In this section, we give an explicit formula for the Hessian of $\log P_t$ for a wide class of diffusion operators.
We need to introduce some additional notation. For $a,\sigma> 0$, consider the general Ornstein-Uhlenbeck operator $L^{\text{ou}}_{\sigma,a}$ on $\R^n$
 \[
 L^{\text{ou}}_{\sigma,a}=\frac{1}{2} \sigma^2\Delta - ax \cdot \nabla,
 \]
where the dot stands for the scalar product. Observe that the Ornstein-Uhlenbeck operator given in the introduction corresponds to $\sigma=\sqrt{2}$ and $a=1$. In what follows, we will write $L^{\text{ou}}$ instead of $L^{\text{ou}}_{\sigma,a}$ in order not to overload the notation.
 Let $(B_t)_{t\geq0}$ be a standard Brownian motion on $\R^n$ on a (filtered) probability space $(\Omega,\P)$ which we fix.
 For any $x\in \R^n$ let $(X_s^x)_{s \geq 0}$ be the (unique strong) solution to
\[
X_t^x=x+\sigma  B_t-a \int_0^t X_s^x ds.
\]
This is the so-called Ornstein-Uhlenbeck process (with parameters $a,\sigma$) starting at  $x$ ;  its infinitesimal generator is $L^{\text{ou}}$. For any $t>0$, the law of $X_t^x$ will be denoted by $\gamma_t^x$ and is given by the (general) Mehler formula
\[
d\gamma_t^x(y) = \frac{1}{Z_t}M_t(x,y) \,dy
\]
with 
\[
M_t(x,y)=M_t^{\sigma,a}(x,y)=\exp\left(-\f {a|y-e^{-at} x|^2} {\sigma^2(1-e^{-2at})}  \right),\qquad y \in \R^n,
\]
and $Z_t$ a normalizing constant. We will denote by $\gamma$ the equilibrium measure of the process given by 
\[
\gamma(dy) = \frac{1}{Z} \exp\left(-\f {a|y|^2} {\sigma^2}  \right)\,dy,\qquad Z = \left(\frac{\pi \sigma^2}{a}\right)^{n/2}.
\]
Note that when $a=1$ and $\sigma= \sqrt{2}$, then $\gamma = \gamma_n$ is the standard Gaussian distribution on $\R^n$.

 We also consider the following perturbation of the Ornstein-Uhlenbeck operator
 \[
L^V = L^{\text{ou}} - V 
\]
where $V \colon \mathbb{R}^n \to \mathbb{R}$ is a potential that acts multiplicatively, namely
 $L^V f= L^{\text{ou}}f - Vf$. The associated semi-group will be denoted by $(P_t^V)_{t \geq 0}$.
 We recall that $P_t^V$ can be represented by the Feynman-Kac formula:
 \begin{prop}\label{prop:Feynman-Kac}Suppose that $V:\R^n \to \R$ is continuous and bounded from below and define for $t\geq0$ the operator $P_t^V$ by
  \[
 P_t^Vf(x) = \E\left[f(X_t^x) e^{- \int_0^t V(X_s^x)\,ds}\right],\qquad 
\forall x\in \R^n,\qquad \forall f \in \mathbb{L}^2(\gamma).
 \]
 Then $(P_t^V)_{t\geq0}$ is a semi-group on $\mathbb{L}^2(\gamma)$ with infinitesimal generator $L^V$.
 \end{prop}

In the sequel we will need the following definition. 

\begin{defi} \label{def:Q}
Let $t>0$, $x\in \R$  and let $f\in \mathbb{L}^2(\gamma) \setminus\{0\}$ be a non-negative function.  
We define the probability measure 
$\mathbb{Q}_{f,x}$ on $\Omega$ (which depends also on $t$ and $V$)  by
\[
\mathbb{Q}_{f,x}(\Gamma)=\frac{1}{P_t^Vf(x)} \int_\Gamma f(X_t^x) e^{-\int_0^t V(X_s^x)ds} \,d\mathbb{P} 
\]
and use $\mathbb{E}_{f,x}$ for the expectation with respect to $\mathbb{Q}_{f,x}$.
\end{defi}

The following result gives an explicit representation for the Hessian of $\log P_t^Vf$:

\begin{thm} \label{loghessian}
Suppose that $V \colon \mathbb{R}^n \to \mathbb{R}$ is bounded from below 
and in $\mathcal{D}^{(2)}$. For $t >0$ and $x\in \R^n$, set
\[
A_t^x :=- \int_0^t \nabla V(X_s^x)\frac{\sinh (a(t-s))}{\sinh(at)}\, ds 
+\f{2ae^{-at}} {\sigma^2(1-e^{-2at})} (X_t^x-e^{-at}x).
\]
Let $f\in \mathbb{L}^2(\gamma) \setminus\{0\}$ be non-negative ;  with the notation of  Definition \ref{def:Q}, it holds
\[
\nabla (P_t^Vf)(x)
=  \mathbb{E}_{f,x}  (A_t^x)
\] 
and
\begin{equation} \label{froid}\begin{split} 
&\Hess(\log P_t^Vf)(x) +\f{2ae^{-2at}} {\sigma^2(1-e^{-2at})} \mathrm{Id}\\
&= - \int_0^t \left( \frac{\sinh (a(t-s))}{\sinh(at)}\right)^2 \mathbb{E}_{f,x} (  \Hess V(X_s^x)) ds 
+ \mathbb{E}_{f,x} (A_t^x\otimes A_t^x) - \mathbb{E}_{f,x} (A_t^x)\otimes 
 \mathbb{E}_{f,x} (A_t^x).
\end{split}
\end{equation}
\end{thm}
The notation $\nabla$ denotes the gradient with respect to the standard Euclidean metric (note that the Riemannian metric, intrinsic to the equation, is $\widetilde \nabla = \sigma^2 \nabla$).

The interested reader may find a series of articles on first/second order 
Feynman-Kac formulas for general elliptic diffusions on manifolds in \cite{She91, EL94, MS96,  APT03}. Moreover Hessian estimates can be found in  
\cite{Li16,Li18:2} under general conditions that are non-trivial to check 
(exchanging orders of operators, non-explosion, existence of global smooth flows).
In the proof of Theorem  \ref{loghessian}, we are able to compute the derivatives thanks to 
an  explicit formulation of Ornstein-Uhlenbeck bridge  (which appears to be linear in its initial position) and  the 
introduction of the probability $\mathbb{Q}_{f,x}$ (see  \cite{Li17, Li18:3} for more on elliptic diffusion bridges). 

\begin{rem}
Observe that, when $V \equiv 0$, $a=1$ and $\sigma^2=2$, $P_t^V$ is the Ornstein-Uhlenbeck semi-group. In dimension 1, after a change of variable,  \eqref{froid} reads
\begin{align*}
(\log P_t^V f)'' (x) &= -c_t^2 + \Var_{f,x}(A_t^x) - \int_0^t \alpha_t(s)^2 \mathbb{E}_{f,x}(V''(X_s^x))ds
\\&=u_t'' = 
c_t^2 \left(-1+ \mu_2(x) \right),
\end{align*}
using the notation of Section \ref{ss:OU} (with $\alpha_t(s):=\frac{\sinh (a(t-s))}{\sinh(at)}$).
\end{rem}

\begin{proof}[Proof of Theorem \ref{loghessian}]
Fix $x \in \mathbb{R}$ and $t \geq 0$. According to Proposition \ref{prop:Feynman-Kac}, it holds
\[
P_t^Vf(x)= \mathbb{E}\left( f(X_t^x)e^{-\int_0^t V(X_s^x)ds}\right)
=Z^{-1}\int_{\R^n} f(y) \mathbb{E}\left( e^{-\int_0^t V(X_s^x)ds} | X_t^x =y\right) M_t(x,y) \,dy,
\]
where  $Z=Z_t$ is the normalization constant for $M_t(x,y)$ that does not depend on $x$. 
approximation.

Conditioning on $X_t^x=y$, $(X_s^x)_{0 \leq s \leq t}$ is distributed as the
Ornstein-Uhlenbeck bridge $(Y_s^{x,y})_{0 \leq s \leq t}$, which begins at  $x$ and ends at $y$ at the final time  $t$.  To determine the dependence of  the functions with respect to the variable $x$, we use an explicit 
representation of  
$Y_s\equiv Y_s^{x,y}$ as solution of  the following equation 
\[
dY_s=\sigma dB_s-a Y_sds+ \sigma^2\nabla_x \log M_{t- s}(Y_s, y)ds,
\]
with the initial value $Y_0=x$ and where $\nabla_x$ stands for the derivative with respect to the $x$ variable.
It has a singular drift at the terminal time $t$ and so it is initially defined for $s<t$, and then extended by continuity to $X_s=z$ for $s\ge t$.  
We have
\[
\nabla_x\log M_{t} (x,y) 
=
( y-e^{-at} x) \f { 2ae^{-at} } {\sigma^2( 1-e^{-2at}) } = d_t ( y-e^{-at} x),
\]
(which is a drift  pulling toward $y$), where we set $
d_t:= \f { 2ae^{-at} } {\sigma^2( 1-e^{-2at}) } .$
Thus we get
\begin{equation} \label{eq:OUbridge}
dY_s
=\sigma dB_s+\f{2ay e^{-a(t-s)}}{1-e^{-2a(t-s)}} ds-a\f{1+ e^{-2a(t-s)}}{1-e^{-2a(t-s)}} Y_sds.
\end{equation}
The difference $Y_s^{x,y}- Y_s^{0, y}$ solves a time dependent linear equation and is given, for all $s \in [0,t]$, by
\begin{equation}\label{bridge}
Y_s^{x,y}- Y_s^{0, y}= \alpha_t(s)x, \qquad\text{where}\qquad \alpha_t(s):= \frac{\sinh (a(t-s))}{\sinh(at)}.
\end{equation}
Therefore we have
\begin{align}\label{eq:ptv} 
P_t^Vf(x)
&=Z_t^{-1}\int_{\R^n} f(y) \mathbb{E} \left(e^{-\int_0^t V(Y_s^{x,y})ds} \right) M_t(x,y)dy \nonumber \\
&=
Z_t^{-1} \int_{\R^n} f(y) \mathbb{E} \left(e^{-\int_0^t V(\alpha_t(s)x+Y_s^{0, y})ds} \right) M_t(x,y)dy,
\end{align}
Take $f\in \mathcal{C}_K^\infty$; since  $Y_s^{0,y}$ does not depend on $x$, it holds
\begin{align}\label{bridge2}
\nabla_x  \left(e^{-\int_0^t V(Y_s^{x,y})ds} \right)
 = -e^{-\int_0^t V(Y_s^{x,y})ds}\int_0^t \alpha_t(s) \nabla V(Y_s^{x,y}) ds.
\end{align}
So,
\begin{equation}\label{gradient}
 \begin{split}
\nabla (P_t^Vf)(x)
=& 
- Z_t^{-1} \int_{\R^n} f(y) \mathbb{E} \left( e^{-\int_0^t V(Y_s^{x,y})ds}\int_0^t \nabla V(Y_s^{x,y}) \alpha_t(s)\,ds  \right) \,  M_t(x,y)dy \\
& \quad + Z_t^{-1} \int_{\R^n} f(y) \mathbb{E} \left(e^{-\int_0^t V(Y_s^{x,y})ds} \right) \nabla_x  M_t(x,y) dy.
\end{split}
\end{equation}
Plugging in the expression for $\nabla_x\log M_{t} (x,y) $ and reversing the conditioning process, we see that
\begin{align*}
\nabla (P_t^Vf)(x)
&=-  \int_0^t  \mathbb{E}  \left(f(X_t^x) \, e^{-\int_0^t V(X_s^x)ds}  \nabla V(X_s^x)\right)\alpha_t(s)\,ds \\
& \quad   + d_t\mathbb{E} \left(e^{-\int_0^t V(X_s^x)ds}f(X_t^x)(X_t^x-e^{-at}x)
 \right).
\end{align*}
Therefore, for $0 \leq s \leq t$, 
$
\nabla (\log P_t^Vf)(x) 
=
\mathbb{E}_{f,x}  (A_t^x)
$ which proves the first identity of the theorem.
In the calculations above, we have taken liberty to differentiate under the integration sign, which holds 
for any smooth functions with compact support. 
Since $Y_s^{x,y}$ is Gaussian and has moments of all order and $|\nabla V|$ growth at most polynomially, if a sequence $f_n \in \mathcal{C}_K^\infty$ converges to $f \in \mathbb{L}^2(\gamma)$ then
the right hand side of the latter converges uniformly. Hence $P_t^Vf$ is differentiable and the identity holds for any $f \in \mathbb{L}^2(\gamma)$.

Using the same conditioning strategy, we can similarly compute the second 
order derivative of $P_t^Vf$, treated as a symmetric matrix.
For this we  go back to (\ref{gradient}) and differentiate under the integral signs: for any $w \in \R^n$,
\begin{align*}
&\<\Hess (P_t^Vf)(x), w\otimes w\>\\
& = 
- \int f(y) \mathbb{E} \left(e^{-\int_0^t V(Y_s^{x,y})ds}  \int_0^t \alpha_t(s)^2  \< \Hess V (Y_s^{x,y}), w \otimes w \>
\;ds \right) M_t(x,y)dy \\
&\quad +\int f(y) \mathbb{E} \left(e^{-\int_0^t V(Y_s^{x,y})ds} \left(  \int_0^t \alpha_t(s) \<\nabla V(Y_s^{x,y}), w\> 
\;ds \right)^2\right) M_t(x,y)dy\\
& \quad - 2\int f(y)  \mathbb{E} \left(e^{-\int_0^t  V(Z_s^{x,y})ds} \int_0^t \<\nabla V(Y_s^{x,y}),w\>\alpha_t(s)ds\,  \right) 
\<\nabla_x M_t(x,y), w\>dy \\
& \quad +\mathbb{E} \left(e^{-\int_0^t V(X_s^x)ds}f(X_t^x) \<\mathrm{Hess}_{x} M_t(x,z), w\otimes w\>dy\right) .
\end{align*}
The differentiation procedure holds for $f$ in $\mathcal{C}_K^\infty$, and the same approximation argument as before shows that it holds also for any $f\in \mathbb{L}^2 (\gamma)$. 

Next we observe that the following identity holds
$
\mathrm{Hess}_x \log M_t(x,y) 
= 
- d_t e^{-at}  \mathrm{Id}$
where $\mathrm{Id}$ is the $n\times n$ identity matrix.
Therefore,
\begin{align*}
\f {\Hess_x M_t(x,y)} {M_t(x,y)} 
& =
\mathrm{Hess}_x  \log M_t(x,y) +\nabla_x \log M_t(x,y) \otimes \nabla_x \log M_t(x,y) \\
& =
-d_te^{-at}\mathrm{Id}+d_t^2(  y-e^{-at} x)  \otimes(  y-e^{-at} x).
\end{align*}
Using (\ref{bridge}) and (\ref{bridge2}) we get 
\begin{align*}
& \frac{\<\Hess (P_t^Vf)(x), w\otimes w\> }{P_t^Vf(x)}  \\
& = 
- \int_0^t \alpha_t(s)^2 \mathbb{E}_{f,x}  (\< \Hess V(X_s^x), w\otimes w\>)ds 
+ \mathbb{E}_{f,x} \left(\left( \int_0^t \alpha_t(s) \<\nabla V(X_s^x) , w\> ds \right)^2\right)\\
& \quad -
2 \mathbb{E}_{f,x}  \left( \int_0^t\< \nabla V(X_s^x), w\>\alpha_t(s)ds\,d_t  \<X_t^x- e^{-at}x, w\> \right)  +d_t^2  \mathbb{E}_{f,x} \left( \<X_t^x-e^{-at}x, w\>^2\right)\\
& \quad -d_t e^{-at} |w|^2\\
& = 
- \int_0^t \alpha_t(s)^2 \mathbb{E}_{f,x}  (\< \Hess V(X_s^x), w\otimes w\>)ds \\
& \quad + 
\mathbb{E}_{f,x} \left(\left( -\int_0^t \alpha_t(s) \<\nabla V(X_s^x) , w\> ds +d_t  \<X_t-e^{-at}x, w\> \right)^2\right) -d_te^{-at}  |w|^2\\
& = - \int_0^t \alpha_t(s)^2 \mathbb{E}_{f,x}  (\< \Hess V(X_s^x), w\otimes w\>)ds + \mathbb{E}_{f,x} (\<A_t^x, w\>^2)
-d_te^{-at}  |w|^2 .
\end{align*}
We then use the identity
\[
\Hess(\log P_t^Vf)(x) = \frac{\Hess (P_t^Vf)(x) }{P_t^Vf(x)} - \frac{\nabla(P_t^Vf)(x) \otimes \nabla(P_t^Vf)(x)}{P_t^Vf(x)^2} 
\]
to obtain the following 
\begin{align*}
\Hess(\log P_t^Vf)(x) 
=&-d_te^{-at}\mathrm{Id} - \int_0^t \alpha_t(s)^2 \mathbb{E}_{f,x}  ( \Hess V(X_s^x))ds \\
&\quad + \mathbb{E}_{f,x} (A_t^x\otimes A_t^x) - \mathbb{E}_{f,x} (A_t^x)\otimes  \mathbb{E}_{f,x} (A_t^x) .
\end{align*}
This completes  the proof.
\end{proof}

\begin{rem}
Using Equation \eqref{eq:OUbridge} we see that the Ornstein-Uhlenbeck starting from $x$ conditioned to reach $z$ at time $t$ has the following explicit representation:
\begin{equation*}
Z_s^{x, z}=\alpha_t(s) x+  z \int_0^s \f {a\;  \sinh (a(t-s))}{\sinh^2(a(t-r))}dr+\sigma \int_0^s \f { \sinh (a(t-s))}{\sinh(a(t-r))}  dB_r.
\end{equation*}
The one dimensional case can be found in \cite{Don90}, see also \cite{BK13} and the reference therein.  
\end{rem}

\begin{rem} \label{rem}
In some situations,  it might be also useful to control the second order derivative of the semi-group by the derivatives of $f$ themselves.
For example, in \cite{GRS11:2,GRS13}, the authors deal with log-semi-convex functions in order to get a characterization of transport inequalities.
The result below shows how such a log-semi-convexity transfers to the semi-group. More precisely, assume that, in addition to the hypotheses of the theorem, 
$f \colon \R^n \to \R$ is in $\mathcal{D}^{(2)}$ then, with the notation in  the proof of the theorem, it also holds
\begin{align} \label{froid2}
\mathrm{Hess}(\log P_t^V f) (x) 
 =\, & e^{-2at} \mathbb{E}_{f,x} ( \mathrm{Hess}(\log f)(X_t^x) )  - \int_0^t \mathbb{E}_{f,x}(\mathrm{Hess}(V)(X_s^x))e^{-2as}ds\\
 &+ 
\mathbb{E}_{f,x}(\tilde{A}_t^x \otimes \tilde{A}^x_t) 
 - \mathbb{E}_{f,x}(\tilde{A}_t^x) \otimes \mathbb{E}_{f,x}(\tilde{A}_t^x)   \nonumber
\end{align}
where $\tilde{A}_t^x:=\int_0^t \nabla V(X_s^x)e^{-as} ds$.

 Observe that $X_s^x=e^{-as}x + \sigma \int_0^s e^{-a(s-r)} dB_r$ so that, using the Feynman-Kac formula, 
\[
P_t^Vf(x) =  \mathbb{E} \left( f\left(e^{-at}x + \sigma \int_0^t e^{-a(t-r)} dB_r \right) e^{-\int_0^t V(e^{-as}x + \sigma \int_0^s e^{-a(s-r)} dB_r)ds} \right) 
\]
it holds
\[
\nabla P_t^Vf(x)
= 
e^{-at} \mathbb{E}\left(\nabla f(X_t^x) e^{-\int_0^t V(X_s^x)ds} \right) 
- 
\mathbb{E}\left(  f(X_t^x) \int_0^t \nabla V(X_s^x)e^{-as}ds e^{-\int_0^t 
V(X_s^x)ds} \right) .
\]
Differentiating one more time, we get
\begin{align*}
\mathrm{Hess} (P_t^Vf)(x) 
& =
e^{-2at} \mathbb{E}\left( \mathrm{Hess}(f)(X_t^x)e^{-\int_0^t V(X_s^x)ds} 
\right) \\
& \quad -
2 \mathbb{E}\left(  \nabla f(X_t^x) \otimes \int_0^t \nabla V(X_s^x)e^{-as}ds e^{-\int_0^t V(X_s^x)ds} \right) \\
& \quad 
- \mathbb{E}\left(f(X_t^x) \int_0^t \mathrm{Hess}V(X_s^x)e^{-2as}ds e^{-\int_0^t V(X_s^x)ds} \right) \\
& \quad +
\mathbb{E}\left(  f(X_t^x) \int_0^t \nabla V(X_s^x)e^{-as}ds \otimes \int_0^t \nabla V(X_s^x)e^{-as}ds \;e^{-\int_0^t V(X_s^x)ds} \right)
\end{align*}
from which the expected result follows.
\end{rem}
 
Thanks to the result of the previous section and with the help of the $h$-transform, we can obtain explicit formula for the log-Hessian of 
general diffusion semi-groups. 

Given $W \colon \mathbb{R}^n \to \mathbb{R}$, smooth enough, and the operator 
$L^V=\Delta - x \cdot \nabla - V$ on $\mathbb{L}^2(\gamma_n)$,
we define the operator $\mathcal{L}^W$ on $\mathbb{L}^2(e^{W/2}\gamma_n)$ 
by the unitary transform 
($h$-transform) below:
\begin{align*}
\mathcal{L}^W f &:= e^{-W/2} L^V(e^{W/2}f)\\
& = \Delta f -(x-\nabla W) \cdot \nabla f + \left(\frac{1}{2}\Delta W + 
\frac{1}{4} |\nabla W|^2 - \frac{1}{2} x \cdot \nabla W - V \right) f,
\end{align*}
whose associated semi-group $\mathcal{P}_t^W$ is intertwined with 
 $P_t^V$ by
 \[
 \mathcal{P}_t^W f = e^{-W/2} P_t^V (e^{W/2}f).
 \]
$\R^n$.
Let $h: \R^n\to \R$ and define $d\mu_h(x)=e^{-h}dx$.
If we seek a representation for the reversible operator
\[
L_h := \Delta - \nabla h \cdot \nabla
\]
on $\mathbb{L}^2(\mu_h)$ of the form $L_h f=\mathcal{L}^Wf=e^{-W/2} L^V(e^{W/2}f)$, 
we choose $W$ and then $V$ so that
$$
\nabla \left( \frac {|x|^2} 2- W\right)  = \nabla h
\qquad \text{and} \qquad
\frac{1}{2}\Delta W + \frac{1}{4} |\nabla W|^2 - \frac{1}{2} x \cdot \nabla W - V = 0 .
$$
A  function  $f$ belongs to $\mathbb{L}^2(\mu_h)$  if and only if  $fe^{W/2}$ belongs to $\mathbb{L}^2(\gamma_n)$.
We  denote by $(\mathcal{P}_t)_{t \geq 0}$ the semi-group associated to the diffusion operator $L_h:=\Delta - \nabla h \cdot \nabla$. The operator $L_h$ is essentially self-adjoint on $\mathcal{C}_K^\infty$, see \cite{Li92:phd}.
Theorem \ref{loghessian} and Remark \ref{rem} (with $a=1$ and $\sigma = 
\sqrt{2}$) then admits the following immediate corollary.
\begin{cor} \label{cor:loghessian}
Let $h \colon \mathbb{R}^n \to \mathbb{R}$ belongs to $\mathcal{D}^{(4)}$. Set $\mu_h(dx)=e^{-h(x)}dx$ and
$$
W(x):=\frac{|x|^2}{2} - h(x), \quad \qquad 
V(x):= \frac{1}{2}(n-\Delta h) - \frac{1}{4}(|x|^2-|\nabla h|^2), \quad 
x \in \mathbb{R}^n.
$$ 
Assume that $V$ is bounded from below. Let $f \in \mathbb{L}^2(\mu_h) \setminus \{0\}$ be non negative and, for all $x \in \R^n$, denote by $\E_{f,x}^W$ the expectation with respect to  the probability measure $\mathbb{Q}_{e^{\frac W 2} f, x}$ introduced in Definition \ref{def:Q}.
at $x$. 
Then
\begin{align}  \label{cor:froid}
\mathrm{Hess}(\log \mathcal{P}_t f)(x) 
& = 
  - \frac{1}{2}(\mathrm{Id}-\mathrm{Hess}(h)(x))  + \mathbb{E}_{f,x}^W (A_t^x\otimes A_t^x) - \mathbb{E}_{f,x}^W (A_t^x)\otimes  \mathbb{E}_{f,x}^W (A_t^x) 
  \nonumber \\
& \quad  -c_t^2\mathrm{Id} - \int_0^t \left( \frac{\sinh (t-s)}{\sinh(t)}\right)^2 \mathbb{E}_{f,x}^W (  \Hess V(X_s^x)) ds 
\end{align}
where
\[
A_t^x:=- \int_0^t \nabla V(X_s^x)\frac{\sinh (t-s)}{\sinh(t)}\, ds 
+\f{e^{-t}} {1-e^{-2t}} (X_t^x-e^{-t}x)
\]
and $c_t$ is given by \eqref{eq:ct}.
Assume in addition $f\in \mathcal{D}^{(4)}$, then for $\tilde{A}_t^x:=\int_0^t \nabla V(X_s^x)e^{-as} ds$, 
the following holds \begin{align} \label{cor:froid2}
\mathrm{Hess}(\log \mathcal{P}_t f) (x) 
& = - \frac{1}{2}(\mathrm{Id}-\mathrm{Hess}(h)(x)) + e^{-2at} \mathbb{E}_{f,x}^W ( \mathrm{Hess}(\log f)(X_t^x) ) \\ 
& \quad +  \mathbb{E}_{f,x}^W(\tilde{A}^x_t \otimes \tilde{A}^x_t) - \mathbb{E}_{f,x}^W(\tilde{A}^x_t) \otimes \mathbb{E}_{f,x}^W(\tilde{A}^x_t) 
 - \int_0^t \mathbb{E}_{f,x}^W(\mathrm{Hess}(V)(X_s^x))e^{-2as}ds.  \nonumber
\end{align}
\end{cor}

Observe that, as for Theorem \ref{loghessian}, Corollary \ref{cor:loghessian} contains the case of the Ornstein-Uhlenbeck semi-group which corresponds to the trivial case $W=V=0$.

\section{The $M/M/\infty$ semi-group}\label{sec:mmi}

In this section we deal with the $M/M/\infty$ queuing process, which is a 
discrete analogue of the Ornstein-Uhlenbeck process on the integers $\mathbb{N}:=\{0,1\dots\}$. 
First we obtain lower bounds of $\Delta \log P_t f$,  where $\Delta$ is the discrete Laplacian. Then, 
we investigate  the deviation property of log-semi-convex functions and prove that such a property, contrary to the continuous setting, 
does not hold unless the function is log-convex. In the last subsection, we prove that the Talagrand Regularization effect in $\mathbb{L}^1$ holds 
by means of the strategy of the uniform bound on $P_t$ presented in the introduction. We start with the notation.

\subsection{Notation and setting}\label{Sec:notation}
In all what follows, we will deal with the following classical probability distributions on $\N$ :
\begin{itemize}
\item $\mathcal{B}(n,p)$ stands for the binomial probability measure of parameters $n \in \mathbb{N}$ and $p \in [0,1]$, with the convention that $\mathcal{B}(n,0)=\delta_0$ (the Dirac mass at $0$) and $\mathcal{B}(n,1)=\delta_n$. When $n=1$, we simply denote by $\mathcal{B}(p)$ the Bernoulli distribution of parameter $p$.
\item $\mathcal{P}( \theta)$ stands for the Poisson probability measure of intensity $\theta$  whose probability distribution function will be denoted by $\pi_\theta$ and is given by  $\pi_\theta(k)=e^{-\theta} \theta^k/k!$, $k \in \mathbb{N}$. At some points, we will make a slight abuse of notation and write $\pi_\theta (A) = \sum_{a \in A} \pi_\theta(a)$, for $A \subset \N.$
\end{itemize}

The $M/M/\infty$ queuing  process is defined through its infinitesimal generator $L$, acting on functions on the integers as
\begin{equation}\label{eq:MMgen}
Lf(n):= n\mu \left[f(n-1)-f(n)\right] + \lambda \left[f(n+1)-f(n)\right], \quad n \in \mathbb{N}
\end{equation}
where $\lambda,\mu>0$ are fixed parameters. In the above expression, there is no need to define $f(-1)$ since it is multiplied by $0$. 
We use the following notation for the discrete derivative:
\[
Df(n):=f(n+1)-f(n), \qquad n \in \mathbb{N} ,
\]
and for the discrete second order derivative (Laplacian):
\begin{equation} \label{laplacian}
\Delta f(n) := f(n+1)+f(n-1)-2f(n) = D(Df)(n-1), \qquad n \in \mathbb{N}\setminus\{0\} .
\end{equation}
Then $Lf(n)=\lambda \Delta f(n)-(n\mu -\lambda) Df(n-1)$.

Denote by $(X_t)_{t \geq 0}$ the Markov (jump) process associated to $L$, 
so that for all (say) bounded function $f$ it holds $P_tf(n)=\mathbb{E}(f(X_t) | X_0=n)$, $n\in \N.$
A remarkable feature of the $M/M/\infty$ queuing process is that
\[
\mathcal{L}(X_t | X_0=n)= \mathcal{B}(n,p(t)) \star \mathcal{P}(\rho q(t) )
\]
where $\star$ stands for the convolution,
\[
p(t):= e^{-\mu t}, \quad q(t)=1-p(t), \quad \rho = \frac{\lambda}{\mu}.
\]
 In other words
\[
P_tf(n)=\mathbb{E}(f(Y_t+Z_t))
\]
with $Y_t \sim \mathcal{B}(n,p(t))$ independent of $Z_t \sim \mathcal{P}(\rho q(t) )$ which can be seen as an analogue of the Mehler Formula \eqref{eq:ou} for the Ornstein-Uhlenbeck semi-group.

Finally, we recall that the $M/M/\infty$ queuing process is reversible with respect to the Poisson measure $\mathcal{P}(\rho)$.

In the next section we deal with estimates on the Laplacian of $\log P_tf$.

\subsection{{Log-semi-convexity} of the queuing process}
\label{ss:mminf-slc}

In this section we investigate the behavior of $\Delta \log P_tf$. 
The main result of the section (Proposition \ref{prop:pt}) is that for any starting function $f$ on the integers, 
as for the Ornstein-Uhlenbeck semi-group, $\Delta \log P_tf$ is bounded below by some universal constant depending only on $t$ 
and on the parameters of the process (namely $\lambda$ and $\mu$), but not on $f$. 

\begin{prop} \label{prop:pt}
Let $f \colon \mathbb{N} \to \mathbb{R}^+$ not identically vanishing. Then for all $t>0${ ,}
\begin{equation}\label{eq:prop:pt}
\Delta \log P_tf(n) \geq \log \left(\frac{1}{12} \left( 1-\frac{p^2}{(p+\rho(1-p)^2)^2}\right)\right) \qquad n=1,2 \dots 
\end{equation}
with $p=p(t)=e^{-\mu t}$ and $\rho=\lambda/\mu$.
\end{prop}

\begin{rem}
Notice the right hand side of \eqref{eq:prop:pt} tends to $-\infty$ when $t \to 0^+$, as it should be,
since $f$ can be any function. On the other hand, the right hand side of  
\eqref{eq:prop:pt} tends to $-\log (12)$ as $t$ tends to $\infty$.
This comes from the technicality of the proof, we believe however that there should exist a lower bound on 
$\Delta \log P_tf(n)$ that tends to $0$ as $t$ tends to infinity. 
\end{rem}

The proof of Proposition \ref{prop:pt} relies on the following lemma which asserts that a positive combination of log-convex (or more generally log-semi-convex) functions is log-convex (log-semi-convex). 

\begin{lem} \label{lem:combination}
Let $f_i \colon \mathbb{N} \to (0,\infty)$, $i=1,\dots,N$, be a family of positive functions, with $N$ possibly infinite. 
Assume that for all $i$  and all $n=1,\dots$, $\Delta \log f_i (n)\geq -\beta_i$ for some $\beta_i \in \mathbb{R}$. Then, for all $\alpha_1,\dots,\alpha_N>0$, 
\[
 \Delta \log \left( \sum_{i=1}^N \alpha_i f_i \right) \geq - \max_{1\leq i\leq N}  \beta_i.
\]
\end{lem}
The continuous counterpart of this result is classical and could be used to prove this discrete statement. For the sake of completeness we give below a direct proof. 
\begin{proof}[Proof of Lemma \ref{lem:combination}]
By induction, and possibly taking the limit, it suffices to prove the result for $N=2$. Moreover, by homogeneity we can assume without loss of generality that $\alpha_1=\alpha_2=1$. Replacing $\beta_1,\beta_2$ by $\max(\beta_1,\beta_2)$ one can also assume that $\beta_1=\beta_2=\beta \in \R$. Let $f,g \colon \mathbb{N} \to (0,\infty)$ be two positive functions with
$\Delta \log f \geq -\beta$ and $\Delta \log g \geq -\beta$. 
Since $u_\beta(n) := \beta n^2/2$, $n \in \N$, satisfies $\Delta u_\beta (n) = \beta$ for all $n\geq 1$, putting $\tilde{f} := fe^{u}$ and $\tilde{g} := ge^{u}$, it is enough to prove that $\Delta \log (\tilde{f}+\tilde{g}) \geq0$. The assumption $\Delta \log \tilde{f} \geq0$ and $\Delta \log \tilde g \geq0$ guarantees that, for all $n\geq1$, it holds
\[
\tilde{f}(n) \leq \sqrt{\tilde{f}(n+1)} \sqrt{\tilde{f}(n-1)}\qquad \text{and}\qquad\tilde{g}(n) \leq \sqrt{\tilde{g}(n+1)} \sqrt{\tilde{g}(n-1)}.
\]
Adding these inequalities, and applying Cauchy-Schwarz, yields to
\[
\tilde{f}(n)+\tilde{g}(n)  \leq \sqrt{\tilde{f}(n+1)+\tilde{g}(n+1) } \sqrt{\tilde{f}(n-1)+\tilde{g}(n-1)},
\]
which shows that $\Delta \log (\tilde{f}+\tilde{g}) \geq0$ and completes the proof.
\end{proof}

\begin{proof}[Proof of Proposition \ref{prop:pt}]
Fix $t > 0$; we have
\[
P_tf(n) = \mathbb{E}(f(X_t) | X_0=n) = \sum_{k=0}^\infty f(k) \mathbb{P}(X_t = k |X_0=n),\qquad \forall n \in \N.
\]
For all $k \in \N$, denote by $F_k(n) = \P(X_t = k | X_0=n)$, $n \in \N$.
According to Lemma \ref{lem:combination}, it is enough to show that for all $k \in \N$, it holds
\begin{equation}\label{eq:toprove}
\Delta \log F_k(n) \geq \log \left(\frac{1}{12} \left( 1-\frac{p^2}{(p+\rho(1-p)^2)^2}\right)\right),\qquad \forall n\geq 1.
\end{equation}
Since $P_t$ is reversible with respect to $\mathcal{P}(\rho)$, it holds 
\[
\P(X_t = k | X_0=n) = \pi_\rho(k)\frac{\P(X_t=n | X_0=k)}{\pi_\rho(n)},\qquad n\in \N.
\]
Therefore,
\[
\log F_k(n) = \log \pi_\rho(k) - \log \pi_\rho(n) + \log G_k(n),
\]
where $G_k(n) =\P(Y_t+Z_t = n)$, with as above, $Y_t \sim \mathcal{B}(k,p)$ and $Z_t \sim \mathcal{P}(\rho (1-p))$. 
A simple calculation shows that, for any parameter $\theta>0$, it holds for all $n\geq 1$
\[
\Delta \log \pi_\theta (n) = \log \left( \frac{\pi_\theta(n+1)\pi_\theta(n-1)}{\pi_\theta(n)^2}\right) = \log \frac{(n!)^2}{(n+1)! (n-1)!} = 
\log \frac{n}{n+1}.
\]
From this follows that $\Delta \log F_0(n)=0$, $n\geq 1$, and that for $k\geq 1$, $\Delta \log F_k(n) \geq \Delta \log G_k(n)$, $n\geq 1.$ So it is enough to show that the bound \eqref{eq:toprove} is satisfied by $G_k$.

Let us first treat the case $k=1$ and show the following slightly better lower bound: 
\[
\Delta \log G_1 \geq \log\left(\frac{1}{2} \left( 1-\frac{p^2}{(p+\rho(1-p)^2)^2}\right)\right) := - \alpha
\] 
or equivalently
\begin{equation}\label{eq:G_1}
G_1(n)^2 \leq e^{\alpha}G_1(n+1)G_1(n-1),\qquad \forall n\geq 1.
\end{equation}
For all $n\geq 0$, it holds 
\[
G_1(n) = \left((1-p) + p \frac{n}{\rho(1-p)} \right) \frac{(\rho (1-p))^n}{n!} e^{-\rho (1-p)} ,\qquad n\geq 1.
\]
So, for $n\geq 1$,
\begin{align*}
\frac{G_1(n+1)G_1(n-1)}{G_1(n)^2} &= \frac{n}{n+1} \frac{ \left((1-p) + 
p \frac{n+1}{\rho(1-p)} \right)  \left((1-p) + p \frac{n-1}{\rho(1-p)} \right) }{ \left((1-p) + p \frac{n}{\rho(1-p)} \right)^2}\\
& =  \frac{n}{n+1} \frac{ \left((1-p) + p \frac{n}{\rho(1-p)} \right)^2 
 -  \left(\frac{p}{\rho(1-p)}\right)^2}{ \left((1-p) + p \frac{n}{\rho(1-p)} \right)^2}\\
& \geq \frac{1}{2} \left(1 - \frac{p^2}{\left(\rho (1-p)^2 + p \right)^2}\right)
\end{align*}
and so taking the $\log$ gives the announced lower bound for $\Delta \log 
G_1.$
\begin{rem}
Note that one could be more accurate by keeping the $\frac{n}{n+1}$ factor which eventually yields to the bound 
\[
\Delta \log F_1(n) \geq  \log  \left(1 - \frac{p^2}{\left(\rho (1-p)^2 + p \right)^2}\right),\qquad n\geq1.
\]
\end{rem}
Now let us treat the case $k\geq 2$. It will be convenient to write $Y_t = Y'_t + \varepsilon_t$ with $Y'_t \sim \mathcal{B}(k-1,p)$ and $\varepsilon_t \sim \mathcal{B}(p)$ two independent random variables also independent of $Z_t$.
Conditioning with respect to $Z_t + \varepsilon_t$ and using \eqref{eq:G_1}, we get
\begin{align}\label{eq:inegutile}
G_k(n) & = \sum_{j=0}^{n} \P(Y_t' = j)G_1(n-j)\\\notag
& \leq   
\P(Y_t'  = n)G_1(0)+ e^{\alpha/2}  \sum_{j=0}^{n-1} \P(Y_t'  = j)G_1(n+1-j)^{1/2}G_1(n-1-j)^{1/2}\\\notag
& \leq  \P(Y_t' = n)G_1(0) \\
& \quad +e^{\alpha/2}   \left(\sum_{j=0}^{n-1} \P(Y_t'= j)G_1(n+1-j)\right)^{1/2}  \left(\sum_{j=0}^{n-1} \P(Y_t' = j)G_1(n-1-j)\right)^{1/2}\\\notag
&= \P(Y_t'  = n)G_1(0)
+e^{\alpha/2}   \left(\sum_{j=0}^{n-1} \P(Y_t'= j)G_1(n+1-j)\right)^{1/2}  G_{k}(n-1)^{1/2}
\end{align}
Now let us treat separately the cases : 
\[
\text{(a)} \ n\geq k \geq 2, \qquad  \text{(b)}\  1\leq n \leq k-2, k\geq 
3 \qquad  \text{(c)}\  n=k-1, k\geq 2.
\]

(a) Suppose $n \geq k \geq 2$, then $\P(Y_t'  = n)=0$ and so \eqref{eq:inegutile} yields to
\[
G_k(n) \leq e^{\alpha/2} G_k(n+1)^{1/2} G_k(n-1)^{1/2}.
\]

(b) Fix $k\geq 3$. Let us admit for a moment that there exists $\beta>0$ (independent of $k$) such that for all $1\leq n \leq k-2$, 
\begin{equation}\label{eq:binomineq}
\P(Y_t' = n) \leq e^{\beta/2}\P(Y_t'  = n-1)^{1/2}\P(Y_t'  = n+1)^{1/2},\qquad \forall 1\leq n \leq k-2.
\end{equation}
As we will see below, the optimal $\beta$ is $\log 3$.
If $1\leq n\leq k-2$, then inserting \eqref{eq:binomineq} into \eqref{eq:inegutile} gives
\begin{align*}
G_k(n) & \leq e^{\beta/2}\left(\P(Y_t'  = n+1)G_1(0)\right)^{1/2}\left(\P(Y_t'  = n-1)G_1(0)\right)^{1/2} \\
&+e^{\alpha/2}   \left(\sum_{j=0}^{n-1} \P(Y_t' = j)G_1(n+1-j)\right)^{1/2}  G_{k}(n-1)^{1/2}\\
& \leq e^{\max(\alpha;\beta)/2} \left[\left(\P(Y_t' = n+1)G_1(0)\right)^{1/2} +  \left(\sum_{j=0}^{n-1} \P(Y_t'  = j)G_1(n+1-j)\right)^{1/2}\right] G_k(n-1)^{1/2}\\
& \leq \sqrt{2}e^{\max(\alpha;\beta)/2} G_k(n+1)^{1/2} G_k(n-1)^{1/2},
\end{align*}
where the second inequality comes from $\P(Y_t'  = n-1)G_1(0) \leq G_{k}(n-1)$ and the third inequality follows from $\sqrt{a} + \sqrt{b} \leq \sqrt{2} \sqrt{a+b}$, $a,b \geq0$.
To determine $\beta$ in \eqref{eq:binomineq} note that
\[
\binom{k-1}{n} p^n(1-p)^{k-1-n} \leq e^{\beta/2} \left(\binom{k-1}{n-1} p^{n-1}(1-p)^{k-n}\right)^{1/2} \left(\binom{k-1}{n+1} p^{n+1}(1-p)^{k-n-2}\right)^{1/2}
\]
is equivalent to
\[
\frac{1}{(n! (k-n-1)!)^2} \leq e^{\beta}  \frac{1}{(n-1)! (k-n)!}\frac{1}{(n+1)! (k-n-2)!}
\]
 which is equivalent to
 \[
 \frac{n+1}{n} \leq e^{\beta}  \frac{k-n-1}{k-n},\qquad \forall 1\leq n \leq k-2.
 \]
Observe that
 \[
\frac{ (n+1)(k-n)}{n(k-n-1)}  = 1 + \frac{k}{n(k-1)-n^2}.
 \]
 The minimum value of the function $n\mapsto n(k-1)-n^2$ on $\{1,\ldots, k-2\}$ is $k-2$ (reached at $1$ and $k-2$). So
 $\max_{1\leq n\leq k-2}\frac{ (n+1)(k-n)}{n(k-n-1)} = 1 + \frac{k}{k-2}= 2+ \frac{1}{k-2} \leq 3.$ Therefore, one can take $\beta = \log 3.$

(c) Finally, let us assume that $k\geq 2$ and $n=k-1$.
Let us admit for a moment that
\begin{equation}\label{eq:inegutile2}
\P(Y_t'  = k-1)G_1(0)  \leq  \left(\P(Y_t'  = k-1)G_1(1)\right)^{1/2} 
\left( \P(Y_t' = k-2 ) G_1(0) \right)^{1/2}.
\end{equation}
Then, inserting \eqref{eq:inegutile2} into \eqref{eq:inegutile}, and reasoning exactly as in the case (b) gives 
\[
G_k(k-1) \leq \sqrt{2} e^{\alpha/2} G_{k}(k)^{1/2} G_k(k-2)^{1/2}.
\]
To prove \eqref{eq:inegutile2}, first observe that $\P(Y_t' = k-1) = p^{k-1}$, $\P(Y_t' = k-2) = (k-1)p^{k-2}(1-p)$ and so $\P(Y_t'= k-1) \leq \frac{p}{1-p} \P(Y_t' = k-2).$ Since, $G_1(0) = (1-p)  e^{-\rho (1-p)}$ and $G_1(1)= \left((1-p) + p \frac{1}{\rho(1-p)} \right) (\rho (1-p)) e^{-\rho (1-p)}$, we see that
$G_1(0) = \frac{1}{(1-p)\rho + \frac{p}{1-p}} G_1(1).$ Therefore, 
\[
\P(Y_t'  = k-1)G_1(0) \leq  \frac{p}{(1-p)^2\rho + p}\P(Y_t'  = k-2)G_1(1) \leq \P(Y_t'  = k-2)G_1(1)
\]
which gives \eqref{eq:inegutile2}.

Putting everything together, one gets for all $k\geq 0$ and $n\geq 1$,
\[
\Delta \log G_k(n) \geq - \max(\alpha ;\beta)-\log2 \geq -\alpha - \beta -\log 2 =   \log \left( \frac{1}{12} \left(1 - \frac{p^2}{\left(\rho (1-p)^2 + p \right)^2}\right)  \right)
\]
which completes the proof.
\end{proof}

\subsection{Remarks on the action of the $M/M/\infty$ semi-group on structured functions}
In this section, we collect some more facts about the action of $P_t$ on log-convex (resp. log-concave) functions.
The first statement, which is a simple application of Cauchy-Schwarz inequality, asserts that if $f$ is log-semi-convex, then so is $P_t f$. 
The second statement is due to Johnson \cite{Joh07} and shows that $P_t$ also leaves stable the class of log-concave functions.
\begin{prop}
Let $f$ be a positive function on $\mathbb{N}$ such that, for some $\beta 
\geq 0$ and all $n=1,2\dots$, $\Delta \log f (n) \geq -\beta$. {Then}
$$
\Delta \log P_tf (n) \geq -\beta \qquad n=1,2,\dots, \quad t \geq 0 .
$$
\end{prop}

\begin{proof}
Recall that $P_tf(n)=\mathbb{E}(f(\varepsilon_1+\dots+\varepsilon_n+Z))$ with $\varepsilon_i \sim \mathcal{B}(p)$ i.i.d.\ and 
independent of $Z \sim \mathcal{P}(\rho q )$, $q=1-p$, and similarly for $P_tf(n-1)$ and $P_tf(n+1)$. Hence, 
computing the expectation with respect to the Bernoulli random variables $\varepsilon_n$ and $\varepsilon_{n+1}$ respectively, we have
\begin{align*}
P_tf(n+1) 
& =
p^2 \mathbb{E}(f(\varepsilon_1+\dots+\varepsilon_{n-1}+Z+2)) + 2p(1-p) \mathbb{E}(f(\varepsilon_1+\dots+\varepsilon_{n-1}+Z+1))  \\
& \quad + (1-p)^2 \mathbb{E}(f(\varepsilon_1+\dots+\varepsilon_{n-1}+Z)) 
\end{align*}
and
\begin{align*}
P_tf(n) 
& =
p \mathbb{E}(f(\varepsilon_1+\dots+\varepsilon_{n-1}+Z+1)) + (1-p) \mathbb{E}(f(\varepsilon_1+\dots+\varepsilon_{n-1}+Z))  .
\end{align*}
Letting $X:=\varepsilon_1+\dots+\varepsilon_{n-1}+Z$, we get
\begin{align*}
\Delta \log P_tf (n) 
&= 
\log \left( \frac{P_tf(n+1)P_tf(n-1)}{P_tf(n)^2}\right) \\
& =
\log \left( \frac{p^2 \mathbb{E}(f(X+2))\mathbb{E}(f(X)) + 2pq \mathbb{E}(f(X+1))\mathbb{E}(f(X)) + q^2 \mathbb{E}(f(X))^2}
{p^2 \mathbb{E}(f(X+1))^2 + 2pq \mathbb{E}(f(X+1))\mathbb{E}(f(X)) + q^2 \mathbb{E}(f(X))^2}\right) .
\end{align*}
Now, since $\Delta \log f \geq  -\beta$, we infer that $e^{-\beta/2}f(n) \leq  \sqrt{f(n+1)f(n-1)}$. Therefore, using the Cauchy-Schwarz Inequality,
$$
e^{-\beta}\mathbb{E}(f(X+1))^2 \leq  \mathbb{E}(\sqrt{f(X+2)f(X)})^2 \leq 
 \mathbb{E}(f(X+2)) \mathbb{E}(f(X)) .
$$
Hence
\begin{align*}
 p^2 \mathbb{E}(f(X+2)) & \mathbb{E}(f(X)) + 2pq \mathbb{E}(f(X+1))\mathbb{E}(f(X)) + q^2 \mathbb{E}(f(X))^2 \\
& \geq
p^2 e^{-\beta}\mathbb{E}(f(X+1))^2 + 2pq \mathbb{E}(f(X+1))\mathbb{E}(f(X)) + q^2 \mathbb{E}(f(X))^2 \\
& \geq 
e^{-\beta} \left( p^2 \mathbb{E}(f(X+1))^2 + 2pq \mathbb{E}(f(X+1))\mathbb{E}(f(X)) + q^2 \mathbb{E}(f(X))^2 \right)
\end{align*}
which leads to the desired result.
\end{proof}

Recall that a function $f : \N \to (0,+\infty)$ is said log-concave if $\Delta \log f(n) \leq 0$, for all $n\geq1$, or in other words if 
\[
f(n)^2 \geq f(n-1) f(n+1),\qquad \forall n\geq 1.
\] 
It is said ultra-log concave if $n\mapsto n! f(n)$ is log-concave, or equivalently
\[
f(n)^2 \geq \frac{n+1}{n} f(n-1)f(n+1),\qquad \forall n\geq 1.
\]
It is easily checked that $f$ is ultra-log-concave if $f/\pi_\theta$ is log-concave for some (and thus all) $\theta>0$.

The following result is due to Johnson \cite{Joh07}. 
\begin{thm}Let $(X_t)_{t\geq0}$ be the $M/M/\infty$ process with generator \eqref{eq:MMgen}, associated semi-group $(P_t)_{t\geq0}$ and reversible 
distribution $\pi_\rho.$
For all $t\geq 0$, denote by $h_t$ the distribution function of the law of $X_t$. If $h_0$ is ultra-log-concave, then for all $t>0$, $h_t$ is also 
ultra-log-concave. Equivalently, if $f_0 : \N \to (0,+\infty)$ is log-concave and integrable with respect to $\pi_\rho$ then, for all $t>0$, $f_t = P_t f$ is also log-concave.
\end{thm}
We note that the preservation of ultra-log-concavity by the $M/M/\infty$ process was proved in \cite{Joh07} en route to proving the 
maximum entropy property of the Poisson distribution; related properties connected to Poisson and compound Poisson approximation
may also be found in \cite{JKM13, BJKM10}. For the sake of completeness, we briefly sketch Johnson's proof (see \cite{Joh07} for details).
\begin{proof}[Sketch of proof]
Fix some $t>0$. The proof relies on the following explicit representation 
of $X_t$:
\[
X_t = \sum_{k=1}^{X_0} \varepsilon_k + Z,
\]
where, as in Section \ref{Sec:notation}, the random variables $X_0$, $Z$, 
$\varepsilon_k$, $k\geq 1$, 
are independent, $Z$ has law $\mathcal{P}(\rho (1-p))$ and $\varepsilon_k$ has law $\mathcal{B}(p)$, $k\geq 1$, with $p=p(t) = e^{-\mu t}$. 
According to \cite[Proposition 3.7]{Joh07}, the random variable $ \sum_{k=1}^{X_0} \varepsilon_k$ 
(which corresponds to a thinning of $X_0$) has an ultra-log-concave distribution. 
On the other hand, it is easily checked that $Z$ has also an ultra-log-concave distribution. 
Since the class of ultra-log-concave functions is closed under convolution \cite{Wal76, Lig97, Gur09:1, NO12}, 
we conclude that the distribution function of $X_t$ is ultra-log-concave. 

Finally, observe that if $f_0 \in \mathbb{L}^1(\pi_\rho)$ is a log-concave function such that (without loss of generality) 
$\int f_0\,d\pi_\rho = 1$ and $X_0$ has distribution function $h_0=f_0\pi_\rho$, then $h_0$ is obviously ultra-log-concave 
and so, according to what precedes, the distribution function $h_t$ of $X_t$ is also ultra-log-concave. 
Since $P_t$ is reversible with respect to $\pi_\rho$, it holds $h_t = (P_tf_0) \pi_\rho$. And so $f_t=P_tf_0$ is log-concave.
\end{proof}

This result may be seen as a discrete analogue of the preservation of log-concavity by the heat flow (observed by Brascamp and Lieb \cite{BL76a})
and of strengthened versions of this property (observed by Ishige, Salani 
and Takatsu \cite{IST19}).

\subsection{Deviation bounds for log-convex functions - absence of deviation bounds for log-semi-convex functions}
In this section, we investigate deviation bounds of the type $\pi_\theta(\{n : f(n) \geq t \int f\,d\pi_\theta\})$ for 
log-convex, and more generally log-semi-convex, functions $f$.
In other words, we address the analogue of Problem (2) in the introduction for the Poisson distributions.
As our results will reveal, in this discrete setting, an analogue of Problem (2) does not hold in general, but it
does hold if $f$ is assumed to be log-convex.
One reason for this spurious effect is that the tail of the measure $\sum_{k \geq n} \pi_\theta(k)$, in discrete, is of the same order as $\pi_\theta(n)$, 
\textit{i.e.}, with no extra factor, while in the continuous,
$\int_s^\infty e^{-t^2}dt \sim_{s \to \infty} \frac{e^{-s^2}}{2s}$.

In all what follows, will make a frequent use of a non-asymptotic version 
of Stirling formula.
More precisely, the following inequalities for the factorial are known (see \cite{Rob55}) to hold
\[
\sqrt{2 \pi} n^{n+\frac{1}{2}} e^{-n + \frac{1}{12n+1}} < n! < \sqrt{2 \pi} n^{n+\frac{1}{2}} e^{-n + \frac{1}{12n}}, \qquad n \geq 1 .
\]
Hence, 
\begin{equation} \label{Stirling}
n^{n+\frac{1}{2}} e^{-n} \leq n! \leq 3 n^{n+\frac{1}{2}} e^{-n}
\end{equation}
for $n \geq 1$ (since $\sqrt{2 \pi}e^{1/12n} \leq 3$).

Let us begin with a precise tail bound for the Poisson distributions.
\begin{lem}\label{lem:Ptail}
Let $\theta>0$ and define $\Phi_\theta(x):=x \log x -x \log \theta - x + \theta$, $x \geq 1$. Set $F_\theta(u):=\pi_\theta( [u,\infty))$ for the tail of the distribution function of $\pi_\theta$.
For $u \geq 2\theta$, we have
\[
F_\theta(u) \leq \frac{2}{\sqrt{u}} \exp \left\{-\Phi_\theta(u) \right\}.
\]
\end{lem}
\begin{proof} If $u\geq 2\theta$, 
\begin{align*}
F_\theta(u) 
& = 
\sum_{k \geq u} \frac{\theta^k e^{-\theta}}{k!} 
= 
\frac{\theta^u e^{-\theta}}{u!}  \sum_{k \geq u}\frac{\theta^{k-u}}{k(k-1)\dots(u+1)} \\
& \leq 
\frac{\theta^u e^{-\theta}}{u!}  \sum_{k \geq u} 2^{u-k} 
 =
2 \frac{\theta^u e^{-\theta}}{u!} 
\leq 
\frac{2}{\sqrt{u}} \exp \left\{-\Phi(u) \right\} 
\end{align*}
where we used \eqref{Stirling}.
\end{proof}

\begin{prop}\label{prop:poi-lc}
For any $\theta>0$, there exists a constant $c$ that depends only on $\theta$ such that for all $t \geq 4$ 
and all  positive functions $f$ on the integers satisfying $\Delta \log f 
\geq 0$, we have
\[
\pi_\theta(\{f \geq t \int f d\pi_\theta  \}) \leq c\frac{\sqrt{\log \log 
t}}{t \sqrt{\log t}} .
\]
\end{prop}

\begin{proof}
We assume without loss of generality that $\int f d\pi_\theta = 1$ and we follow \cite{GMRS17}. Define $\widetilde{f} \colon [0,\infty) \to (0,\infty)$ as the piecewise linear interpolation of $f$. Since  $\Delta \log 
f \geq 0$, $\log \widetilde{f}$ is convex so that
$\log \widetilde{f}(x)=\sup_{y \geq 0}\{xy-\widetilde{g}(y)\}$, $x \geq 
0$, where $\widetilde{g}(y)=(\log \widetilde{f})^*(y)
= \sup_{x \geq 0}\{yx-\log \widetilde{f}(x)\}$, $y\geq 0$, is the Legendre transform of $\log \widetilde{f}$.
Then, since, for any $n \in \mathbb{N}$ and any $y\geq 0$, $\log f(n) = 
\log \widetilde{f}(n) \geq  ny-\widetilde{g}(y)$, we have
\[
1 = 
\int f d\pi_\theta 
\geq 
e^{-\widetilde{g}(y)} \int e^{ny} d\pi_\theta(n) = \exp\{ -\widetilde{g}(y) +\theta(e^y-1) \} .
\]
Therefore
\[
\widetilde{g}(y) \geq  \theta(e^y-1),\qquad y\geq0,
\]
and in turn 
\begin{align*}
\log f(n) 
& = 
\log \widetilde{f}(n) \leq \sup_{y\geq 0} \{ny - \theta(e^y-1) \}
=
\begin{cases}
n(\log n - \log \theta) - n + \theta & \mbox{if } n \geq \theta \\
0 & \mbox{if } n <  \theta.
\end{cases} \\
& =
\max[n(\log n - \log \theta) - n+\theta, 0] .
\end{align*}
Hence, for $t \geq e^{\theta-1}/\theta$, 
\begin{align*}
\pi_\theta(\{f \geq t \}) 
& \leq 
\pi_\theta(\{ \max(n(\log n - \log \theta) - n+\theta,0) \geq \log t \}) \\
& = 
\pi_\theta(\{n(\log n - \log \theta) - n + \theta \geq \log t \}) = \pi_\theta(\{\Phi_\theta(n) \geq \log t \}) \\
& =
\pi_\theta (\{n \in \mathbb{N}: n \geq \Phi_{\theta}^{-1}(\log t) \})
\end{align*}
where we set $\Phi_\theta(x):=x \log x -x \log \theta - x + \theta$, $x 
\geq 1$ and denoted by $\Phi_\theta^{-1}$ its inverse function which is increasing on $[\theta-1-\log \theta,\infty)$. 
Using Lemma \ref{lem:Ptail}, we get for $t \geq c_\theta$, for some constant depending only on $\theta$,
\[
\pi_\theta (\{n \in \mathbb{N}: f(n) \geq t \}) 
\leq
2 \frac{e^{-\Phi_\theta(\Phi_\theta^{-1}(\log t))}}{\sqrt{\Phi_\theta^{-1}(\log t)}}
=
\frac{2}{t \sqrt{\Phi_\theta^{-1}(\log t)}} .
\]
To end the proof it suffices to observe that
$\Phi_\theta(x/\log x) = x - x[\log \log x + \log \theta + 1]/\log x + \theta \leq x$ for $x$ large enough so that
$\Phi_\theta^{-1}(x) \geq   x/\log x$ (for $x$ large).
\end{proof}

\begin{rem}\label{rem:UPEMLV}
Let us note that the bound in Proposition~\ref{prop:poi-lc} is of optimal 
order.
Indeed, consider the function $f_\lambda(n)=e^{\lambda n}c(\lambda)$, where $c(\lambda)=\exp\{1-e^\lambda\}$, $\lambda\geq0$,
is taken to be the normalizing constant that makes $\int f_\lambda d\pi_1 
 =1$.
Observe that $\Delta \log f_\lambda =0$ since $\log f_\lambda$ is linear.
Now 
\[
\pi_1 ( \{f_\lambda \geq t\} )= \pi_1 \left(\left[\frac{1}{\lambda} \log \left(\frac{t}{c(\lambda)}\right),\infty \right) \right).
\]
We are interested in lower bounds on this Poisson tail.  Let us take $\lambda = \log k$ and  $t= ek^ke^{-k}$, for some integer $k$, so that $\frac{1}{\lambda} \log \left(\frac{t}{c(\lambda)}\right)= k.$ Observe that, using \eqref{Stirling}, 
\[
\pi_1 ([k,+\infty)) \geq \frac{1}{ek !} \geq \frac{1}{3e}k^{-k-\frac{1}{2}}e^k.
\]
Therefore, after some calculations, we get
\[
\frac{t \sqrt{\log t}}{\sqrt{\log \log t}} \pi_1(\{f_\lambda \geq t  \}) \geq  \frac{1}{3} \left(\frac{1+k\log k-k}{k\log(1+k\log k-k)}\right)^{1/2}
\]
and the right hand side goes to $1/3$ as $k \to \infty$, which proves optimality.
\end{rem}

The next proposition goes in the opposite direction to Proposition~\ref{prop:poi-lc}. It states that the log-semi-convex 
property is not enough to ensure a deviation bound better than just Markov's inequality. 
In what follows, $\theta>0$ is fixed and we define for all $\beta\geq0$
\[
\mathcal{F}_\beta :=\{ f\colon \mathbb{N} \to \mathbb{R} \mbox{ such that } \Delta \log f \geq - \beta \mbox{ and } \int f d\pi_\theta=1 \} .
\]
\begin{prop} \label{prop:ce-mmi}
For all $\beta >0$, the following holds
\[
\limsup_{t \to \infty} t \sup_{f \in \mathcal{F}_\beta} \pi_\theta(\{ n : 
f(n) \geq t\}) >0 .
\]
\end{prop}

\begin{proof}
For  $a \geq 0$, define $f_a$ as
\[
f_a (n)=\exp\{-\frac{\beta}{2}(n-a)^2 + Z(a)\}, 
\quad  n \in \mathbb{N}, 
 \mbox{ with } 
Z(a):= - \log \int \exp\{-\frac{\beta}{2}(n-a)^2\} d\pi_\theta(n) 
\]
so that $\int f_a d\pi_\theta =1$. Moreover 
\[
\Delta \log f (n) = -\frac{\beta}{2} \left( (n+1-a)^2 + (n-1-a)^2 -2(n-a)^2 \right)  = - \beta .
\]
Hence, for all $a \geq 0$, $f_a \in \mathcal{F}_\beta$. The expected result will follow if we are able to prove that there exists
$T \colon [0,\infty) \to \mathbb{R}^+$ with $T(a) \to \infty$ as $a \to \infty$ such that
\begin{equation} \label{eq:T(a)}
\limsup_{a \to \infty} T(a) \pi_\theta(\{f_a \geq T(a)\}) > 0 
\end{equation}
since clearly
$\limsup_{a \to \infty} T(a) \pi_\theta(\{f_a \geq T(a)\})  \leq \limsup_{t \to \infty} t \sup_{f \in \mathcal{F}_\beta} \pi_\theta(\{f \geq t\})$.

Set $\Psi_a \colon \mathbb{R}^+ \to \mathbb{R}$, $u \mapsto -\frac{\beta}{2}(u-a)^2 - \log \Gamma(u+1) + u\log \theta - \theta$
where $\Gamma(z):= \int_0^\infty t^{z-1}e^{-t} dt$, $z >0$, is the Gamma functional. It is well known that $\log \Gamma$ is convex on $(0,\infty)$ so that $\Psi_a$ is strictly concave on $\mathbb{R}^+$. Since $\lim_{u 
\to \infty} \Psi_a(u)=-\infty$, this guarantees that  $\Psi_a$ has a unique maximum  on $\mathbb{R}^+$ achieved  at a (unique) point we denote by $u_a \in [0,\infty)$.

We claim that $\mathcal{A}:=\{a \geq 1 \mbox{ such that } u_a \in \mathbb{N} \}$ is infinite and unbounded and $u_a \to +\infty$, as $a \in \mathcal{A}$ tends to $+\infty$. We postpone the proof of the claim and continue with the proof of \eqref{eq:T(a)}.


Set, for $a \in \mathcal{A}$,
\[
T(a) := \exp\left(-\frac{\beta}{2}(u_a-a)^2 + Z(a)\right) .
\]
Now we observe that
\begin{align*}
\pi_\theta(\{f_a \geq T(a)\}) 
& =
\pi_\theta\left(\left \{n : -\frac{\beta}{2}(n-a)^2 \geq -\frac{\beta}{2}(u_a-a)^2  \right\}\right) 
 \geq \pi_\theta(u_a)
=
\frac{\theta^{u_a}e^{-\theta}}{u_a!} .
\end{align*}
Therefore, since $u_a ! = \Gamma(u_a+1)$ for $a \in \mathcal{A}$, 
\[
T(a) \pi_\theta(\{f_a \geq T(a)\})  
\geq 
\exp\{ \log(T(a)) - \log (u_a!) + u_a \log \theta - \theta\} 
= 
\exp\{\Psi_a(u_a) + Z(a)  \} . 
\]
Our aim is to bound from below the right hand side of the latter. We notice that, by definition of $\Psi_a$
and since $n!=\Gamma(n+1)$,
\begin{align*}
\int \exp\{-\frac{\beta}{2}(n-a)^2\} d\pi_\theta(n)
& =
\sum_{n = 0}^\infty \exp\{ -\frac{\beta}{2}(n-a)^2 - \log (n!) +n \log \theta  - \theta \} \\
=
& \sum_{n = 0}^\infty \exp\{ \Psi_a(n) \} .
\end{align*}
Since $\Psi_a'' \leq -\beta$ and $\Psi_a'(u_a)=0$  we have
\[
\Psi_a(n) 
\leq 
\Psi_a(u_a)  + \Psi_a'(u_a)(n-u_a) - \frac{\beta}{2}(n-u_a)^2
=
\Psi_a(u_a) - \frac{\beta}{2}(n-u_a)^2 .
\]
Hence
\[
\int \exp\{-\frac{\beta}{2}(n-a)^2\} d\pi_\theta(n)
\leq 
e^{\Psi_a(u_a)} \sum_{n = 0}^\infty \exp\{ - \frac{\beta}{2}(n-u_a)^2  \}
\leq 
2e^{\Psi_a(u_a)} \sum_{n = 0}^\infty e^{ - \beta n^2/2} .
\]
Setting $c_\beta = -\log \left( 2\sum_{n = 0}^\infty e^{ - \beta n^2/2} \right) $, one gets 
\begin{equation} \label{eq:za}
Z(a) 
= 
- \log \int \exp\{-\frac{\beta}{2}(n-a)^2\}) d\pi_\theta(n)
\geq 
c_\beta - \Psi_a(u_a) .
\end{equation}
The latter implies two useful conclusions. First, for all $a \in \mathcal{A}$, 
\[
T(a) \pi_\theta(\{f_a \geq T(a)\})  
\geq 
\exp\{\Psi_a(u_a) + Z(a)\}
\geq 
e^{c_\beta} .
\]
Second, $\log T(a) = -\frac{\beta}{2}(u_a-a)^2 + Z(a) \geq c_\beta + \log (u_a !) - u_a \log \theta + \theta \to \infty$ as $a \in \mathcal{A}$ tends to infinity.

The desired conclusion follows as soon as we prove the claim above. The equation $\Psi_a'(u_a)=0$ shows that the map 
$a \mapsto u_a$ is continuous. Hence the claim will follow if we can prove that $u_a \to \infty$ as $a$ goes to infinity.
We observe that $\Psi'_a(u)=-\beta(u-a)-\psi(u+1) + \log \theta$ where $\psi(u):=\Gamma'(u)/\Gamma(u)$ is the digamma function, which is increasing on $[1,\infty)$. The following asymptotic is known, $\psi(u) = \log u + o(1)$, as $u$ tends to infinity. Therefore $\Psi_a'(\sqrt{a}) \geq 
\beta(a - \sqrt{a}) - \log (\sqrt{a}) + c > 0$ for $a$ large enough. In particular, for $a$ large enough, $u_a \geq \sqrt{a}$ which proves the claim.
\end{proof}



\subsection{The Talagrand Regularization effect} \label{sec:tal-M/M/infty}

In this section we will prove Talagrand's regularization effect for the $M/M/\infty$ queuing process. This is one of the main result of this paper. We will use the strategy of the uniform bound on $P_t$ presented in the Introduction.
Recall that $\rho=\lambda/\mu$ and that the $M/M/\infty$ semi-group $(P_t)_{t \geq 0}$ is reversible with respect to the Poisson measure $\pi_\rho$ of parameter $\rho$. 
For simplicity we will assume from now on that $\rho=1$. All the results below remain valid for any $\rho>0$, but at the price of more 
technicalities in the proofs, non essential for the purpose of the whole paper. As a motivation, it should be noticed that the 
$M/M/\infty$ semi-group enjoys some sort of hypercontractivity property, see \cite[Section 7]{BT06} (cf. \cite{BHP07, KM06, NPY20}).
Hence the question raised by Talagrand about the regularization property of the semi-group for functions in $\mathbb{L}^1$ 
makes perfect sense. Here is a positive answer.

\begin{thm}[Talagrand's regularization effect for the $M/M/\infty$ queuing process]\label{thm:talagrandmmi}
Let $(P_t)_{t \geq 0}$ be the $M/M/\infty$ semi-group (with $\rho=1$). Then, for every $s >0$, 
there exists a constant $c$  (that depends only on $s$) such that, for all $t \geq 4$,
\[
\sup_{ f \geq 0 : \int f d\pi_1 =1} \pi_1 \left( \{n : P_sf(n) \geq t \} \right) \leq \frac{c\sqrt{\log \log t}}{t \sqrt{\log t}} .
\]
\end{thm}

\begin{rem}
For any fixed $s>0$, this bound is optimal for large values of $t$. Indeed, using the notation of Remark \ref{rem:UPEMLV}, it easily seen that $P_s f_\lambda = f_{\lambda(s)}$, with $\lambda(s) = \log(1+e^{-s}( e^\lambda-1) )$. According to Remark \ref{rem:UPEMLV}, the deviation bound of 
Proposition \ref{prop:poi-lc} is optimal for the family $(f_\lambda)_{\lambda>0}$. Therefore the deviation bound of Theorem \ref{thm:talagrandmmi} 
is also optimal.
\end{rem}

The proof of the theorem is based on an estimate on the following quantity 
\[
\Psi_s(n):= \frac{1}{n!} \sup_{k \geq0} \frac{\mathbb{P}(Y_{n,s} + Z_s = k) }{\pi_1(k)},\qquad s\geq0
\]
where $Y_{n,s}$ is a binomial variable of parameter $n$, $p_s=e^{-t}$ and $Z_s$ is a Poisson variable of parameter $q_s=1-p_s$. 

\begin{lem} \label{lem:preparation}
For all $s>0$, there exists a constant $c$ (that depends only on $s$) such that for any $n \geq 1$,
$\Psi_s(n) \leq \frac{c}{\sqrt{n}}$.
\end{lem}

\begin{rem}
We observe that $1/\sqrt n$ is the correct order.
Indeed, assume that $ne^s \in \mathbb{N}$. Considering the special case $k=n/p_s \in \mathbb{N}$ and then the sole term $j=n$ in the sum
we get
\begin{align*}
\Psi_s(n) 
& \geq 
\frac{1}{n!} \frac{\mathbb{P}(Y_{n,s} + Z_s = n/p_s) }{\pi_1(n/p_s)}
=
e^{1-q_s} \frac{(n/p_s)!}{n!} \sum_{j=0}^{n}  \genfrac{(}{)}{0pt}{}{n}{j} \frac{p_s^j q_s^{n+(n/p_t)-2j}}{((n/p_s)-j)!} \\
& \geq 
e^{p_s}\frac{(n/p_s)!}{n!(nq_s/p_s)!} p_s^n q_s^{nq_s/p_s} .
\end{align*}
Therefore, using \eqref{Stirling}, we have
\begin{align*}
\log \Psi_s(n) 
& \geq 
p_s +\left(\frac{n}{p_s}+\frac{1}{2} \right) \log \left(\frac{n}{p_s} \right) - \frac{n}{p_s}
- \log 3 -\left(n+\frac{1}{2} \right) \log n + n \\
& \quad - \log 3 -\left(\frac{nq_s}{p_s}+\frac{1}{2} \right) \log \left(\frac{nq_s}{p_s} \right) + \frac{nq_s}{p_s} + n \log p_s + \frac{nq_s}{p_s} \log q_s \\
& =
p_s-2\log 3 - \frac{1}{2} \log q_s - \frac{1}{2} \log n
\geq 
- 2 \log 3 - \frac{1}{2} \log n 
\end{align*}
from which we get $\Psi_s(n) \geq \frac{1}{9 \sqrt{n}}$.
\end{rem}

\begin{proof}[Proof of Lemma \ref{lem:preparation}] 
Denoting by $X=(X_t)_{t \geq0}$ the $M/M/\infty$ process, we know that $\mathbb{P}(Y_{n,s} + Z_s = k) = \P(X_s = k |X_0=n).$
Since $\pi_1$ is reversible for $X$, we have 
\begin{equation}\label{eq:rev}
\pi_1(n)\frac{\P(X_s = k |X_0=n)}{\pi_1(k)} = \P(X_s = n |X_0=k)
\end{equation}
and so $\Psi_s(n) = n!\sup_{k\geq0} \P(X_s = n |X_0=k) = n! \sup_{k\geq 0} \P(Y_{k,s} + Z_s = n).$
Using \eqref{eq:rev}, one first sees that if $0\leq k\leq n-1$, then 
\[
 \P(X_s = n |X_0=k) \leq \frac{k!}{n!} \leq \frac{1}{n}.
\]
Now, if $k\geq n$, then using Lemma \ref{binom} below, we see that
\[
\P(Y_{k,s} + Z_s = n) =\sum_{i=0}^{k} \P(Y_{k,s} = i) \P(Z_s = n-i) \leq \sup_{0\leq i \leq k} \P(Y_{k,s} = i) \leq \frac{c_p}{\sqrt{k}} \leq \frac{c_p}{\sqrt{n}},
\]
which completes the proof.
\end{proof}


\begin{lem}\label{binom}
For any $p \in (0,1)$, there exists $c_p>0$ such that 
\begin{equation}\label{eq:binom}
\sup_{0\leq i \leq k} \binom{k}{i} p^i(1-p)^{k-i} \leq \frac{c_p}{\sqrt{k}},\qquad \forall k\geq1.
\end{equation}
\end{lem}
\proof
When $p\in (0,1)$, it is well known that the mode of the binomial distribution $\mathcal{B}(k,p)$ is $i_k:=\lfloor (k+1)p \rfloor$. In other words,
\[
\sup_{0\leq i \leq k} \binom{k}{i} p^i(1-p)^{k-i} = \binom{k}{i_k} p^{i_k}(1-p)^{k-i_k}
\]
Using \eqref{Stirling}, one gets that, when $1\leq i\leq k-1$
\[
\binom{k}{i} p^i(1-p)^{k-i}  \leq 3 \sqrt{\frac{k}{i(k-i)}} \frac{ p^i(1-p)^{k-i}}{\left(\frac{i}{k}\right)^i \left(1 - \frac{i}{k}\right)^{k-i}}\leq 3 \sqrt{\frac{k}{i(k-i)}},
\]
where the last inequality follows from the fact that the function $f(s) = 
s^i(1-s)^{k-i}$, $s \in [0,1]$, reaches its maximum at $s = \frac{i}{k}$. Therefore, if $1\leq i_k\leq k-1$, it holds
\[
\binom{k}{i_k} p^{i_k}(1-p)^{k-i_k}  \leq 3  \sqrt{\frac{k}{i_k(k-i_k)}} \leq c_p' \frac{1}{\sqrt{k}},
\]
for some $c'_p$ depending only on $p$.
Now $i_k = 0$ or $i_k = k$ can only occur if $k \leq \max( \frac{1-p}{p} ;  \frac{p}{1-p}):=k_0$. So letting $c''_p = \sup_{0\leq i \leq k, k\leq k_0} \sqrt{k}\binom{k}{i} p^i(1-p)^{k-i}$, we see that \eqref{eq:binom} holds with $c_p = \max (c'_p ; c''_p)$.
\endproof

With Lemma \ref{lem:preparation} in hand, we are in position to prove Theorem \ref{thm:talagrandmmi}.

\begin{proof}[Proof of Theorem \ref{thm:talagrandmmi}]
We first observe that (strategy of the uniform bound on $P_t$)
\[
\sup_{ f \geq 0 : \int f d\pi_1 =1} \pi_1 \left( \{n : P_sf(n) \geq t \} \right)
\leq 
\pi \left( \left\{n : \sup_{ f \geq 0 : \int f d\pi_1 =1} P_sf(n) \geq t \right\} \right) .
\]
We claim that
\[
\sup_{ f \geq 0 : \int f d\pi_1 =1} P_sf(n) =  \sup_{k\geq 0} \frac{\mathbb{P}(Y_{n,s}+Z_s=k)}{\pi_1(k)} 
\] 
where we recall that $Y_{n,s}$ is a binomial variable of parameter $n$ and $p_s=e^{-s}$ and $Z_s$ is Poisson variable with parameter $q_s=1-p_s$. Indeed  if one considers $f_o=\mathds{1}_{k_o}/\pi_1(k_o)$, for some integer $k_o$, one immediately sees that
\[
\sup_{f \geq 0: \int f d\pi_1=1} P_sf(n) \geq P_sf_o(n) = \sum_{k=0}^\infty f_o(k)\mathbb{P}(X_{n,s}+Y_s=k) = \frac{\mathbb{P}(X_{n,s}+Y_s=k_o)}{\pi_1(k_o)} .
\]
Therefore $\sup_{ f \geq 0 : \int f d\pi_1 =1} P_sf(n) \geq  \sup_{k\geq 0} \frac{\mathbb{P}(X_{n,s}+Y_s=k)}{\pi_1(k)}$. 
On the other hand, for any $f$ non-negative with $\int f d\pi_1 =1$, 
\begin{align*}
P_sf(n) 
& = \sum_{k=0}^\infty f(k)\mathbb{P}(X_{n,s}+Y_s=k) = \sum_{k=0}^\infty f(k) \pi_1(k) \frac{\mathbb{P}(X_{n,s}+Y_s=k)}{\pi_1(k)}  \leq 
\sup_{k\geq 0} \frac{\mathbb{P}(X_{n,s}+Y_s=k)}{\pi_1(k)} 
\end{align*}
which proves the claim.

Recall the definition of $\Psi_s$ right before Lemma \ref{lem:preparation}. From the claim and Lemma 
\ref{lem:preparation}, we have
\[
\sup_{ f \geq 0 : \int f d\pi_1 =1} \pi_1 \left( \{n : P_sf(n) \geq t \} \right)
\leq 
\pi_1\left( \{n : n! \Psi_s(n) \geq t \} \right)
\leq 
\pi_1 \left( \{n : n!/\sqrt{n} \geq t/c \} \right) 
\]
for some constant $c$ depending only on $s$.
Using \eqref{Stirling}, we have $n!/\sqrt{n} \leq 3 \exp \{n\log n -n\}$. 
Hence, setting $H(x):=x \log x - x$, which is an increasing function (hence one to one whose inverse we denote by $H^{-1}$),
\begin{align*}
\sup_{ f \geq 0 : \int f d\pi_1 =1} \pi_1 \left( \{n : P_sf(n) \geq t \} \right) 
& \leq 
\pi_1\left( \{n : e^{H(n)} \geq t/(3c) \} \right) \\
&=
\pi_1 \left( \{ n: n \geq H^{-1}( \log (t/(3c))) \} \right) \\
& =
\pi_1 \left( \{ n:n \geq \lceil H^{-1}( \log (t/(3c))) \rceil \} \right).
\end{align*}
(Here, as usual, $\lceil \cdot \rceil$ denotes the ceiling function, that 
maps $x$ to the least integer greater than or equal $x$).
Next we observe that, according to Lemma \ref{lem:Ptail}, for any integer 
$u \geq 2$,
\[
\pi_1(\{ n: n \geq u\}) 
\leq 2 \frac{e^{-\Phi_1(u)}}{\sqrt{u}}
\leq \frac{e^{-H(u)}}{\sqrt{u}}.
\]
Since $x \mapsto e^{-H(x)}/\sqrt{x}$ is decreasing, for $t$ large enough, 
we end up with
\begin{align*}
\sup_{ f \geq 0 : \int f d\pi_1 =1} \pi_1 \left( \{n : P_sf(n) \geq t \} \right) 
& \leq 
\frac{e^{-H(\lceil H^{-1}( \log (t/(3c))) \rceil)}}{\sqrt{\lceil H^{-1}( \log (t/(3c))) \rceil}}
\leq 
\frac{e^{-H(H^{-1}( \log (t/(3c))))}}{\sqrt{H^{-1}( \log (t/(3c)))}} \\
& =
\frac{3c}{t\sqrt{H^{-1}( \log (t/(3c)))}} .
\end{align*}
Finally, we observe that $H^{-1}(x) \geq \frac{x}{\log x}$ for $x$ large enough, from which the expected result follows.
\end{proof}


\section{Laguerre's semi-groups} \label{sec:Laguerre}

In this section we deal with the Laguerre semi-groups on $(0,\infty)$. We 
may prove that both log-semi-convexity and deviation bounds for log-semi-convex functions (Problem (1) and (2) in the introduction), do not hold. On the other hand, the strategy of the uniform bound on $P_t$ applies and 
will allow us to prove the Talagrand regularization effect for the Gamma probability measures.

In the next subsection, we introduce the Laguerre operator in its full generality.  However, in the subsequent subsection we may, for simplicity, reduce to the sole case $\alpha=3/2$ (see below) which is simpler to handle. Many computations could probably be done for general $\alpha$ but at the price of heavy technicalities. We preferred a simpler presentation rather than a complete one in order to present the phenomenon occurring in the setting of Laguerre's operators.

\subsection{Introduction}
On $(0,\infty)$ denote by $\nu_{\alpha}$, with $\alpha >0$, the Gamma distribution with density
$$
\varphi_{\alpha}(x):= \frac{1}{\Gamma(\alpha)}x^{\alpha-1}e^{- x} , \quad x >0 
$$
with respect to the Lebesgue measure on $(0,\infty)$. It is the reversible measure of the diffusion operator $L^{\alpha}$ (which is negative), called Laguerre operator, defined on smooth enough functions $f$ as
$$
L^{\alpha} f(x)=x f''(x) +(\alpha- x)f'(x), \quad x>0 .
$$
The Laguerre Operator is well-known and related to Laguerre's polynomials
$$
Q_k^{\alpha} (x):= \frac{1}{k!} x^{-\alpha+1}e^{ x} \frac{d^k}{dx^k} \left( x^{k+\alpha-1} e^{- x} \right),
\quad k \in \mathbb{N}, \quad x >0 .
$$
First Laguerre's polynomials are (we omit the super scripts $\alpha$ for simplicity)
$Q_0(x)=1$, $Q_1(x)= \alpha- x$,  $Q_2(x)=\frac{\alpha(\alpha+1)}{2} - (\alpha+1)x + \frac{1}{2} x^2$.
Moreover, the family $(Q_k^{\alpha})_{k \geq 0}$ is an orthogonal decomposition of $L^{\alpha}$
in $\mathcal{L}^2((0,\infty), \gamma_{\alpha})$: namely it is an orthogonal basis of $\mathcal{L}^2((0,\infty), \gamma_{\alpha})$
and each $Q_k^{\alpha}$ is an eigenfunction of $L^{\alpha}$ with associated eigenvalue $-k$, $k=0,1,\dots$.
The associated semi-group, we denote by $(P_t^{\alpha})_{t \geq 0}$, takes the form (see \textit{e.g.}\ \cite{AS64:book})
$$
P_t^{\alpha}f(x) =  \int G_t^{\alpha}(x,y)f(y)d \nu_{\alpha}(y) 
$$
for any $f \in \mathcal{L}^p((0,\infty),\gamma_{\alpha})$ for some $p\geq 
1$, with kernel
$$
G_t^{\alpha}(x,y):=\frac{\Gamma(\alpha)e^{t}}{e^{t}-1}  \left(\frac{e^{t}}{xy} \right)^\frac{\alpha-1}{2} \exp \left\{ -\frac{1}{ e^{t} -1}(x+y)\right\} I_{\alpha -1} \left( \frac{2\sqrt{xye^{t}}}{e^{t}-1} \right) .
$$
Here $I_\beta$ denotes the modified Bessel function of the first kind of order $\beta > -1$, defined as
$$
I_\beta(x):= \sum_{n=0}^\infty \frac{1}{n! \Gamma(n+\beta+1)} \left( \frac{x}{2} \right)^{2n+\beta}, \quad x >0 .
$$
$(P_t)_{t \geq 0}$ is defined as

\subsection{{log-semi-convexity} for the Laguerre semi-groups}

We will prove in this section that there does not exist any uniform lower 
bound (in $x$ and $f$)  on $(\log P_t^\alpha)''(x)$.
For simplicity, and since $I_{1/2}(x) = \sqrt{2/\pi x}\sinh (x)$, $x >0$, is explicit, we may focus only on the case 
$\alpha=3/2$ for which we have (we omit the superscript $\alpha=3/2$ all along this subsection)
$$
P_tf(x) =  \int G_t(x,y)f(y) d\nu(y), \quad 
G_t(x,y):=\frac{\Gamma(\frac{3}{2})e^{t}}{e^{t}-1} \left(\frac{e^{t}}{xy} \right)^\frac{1}{4} \!\!\exp \left\{ -\frac{x+y}{ e^{t} -1}\right\} I_{\frac{1}{2}} \left( \frac{2\sqrt{xye^{t}}}{e^{t}-1} \right) 
$$
with $d\nu(x)=\varphi(x)dx=\frac{\sqrt{x}}{\Gamma(3/2)} e^{-x} dx$. 
Now consider the special test function $f(y)=\delta_y/\varphi(y)$ so that
$P_tf(x)=G_t(x,y)$ and therefore, setting $c_t:=2e^{t/2}/(e^t-1)$, 
$$
\log P_tf(x) 
= 
c_{y,t} - \frac{1}{4} \log x - \frac{x}{e^t-1} + \log I_{\frac{1}{2}} \left( c_t\sqrt{xy} \right) 
=
c'_{y,t}  - \frac{1}{2} \log x - \frac{x}{e^t-1} + \log \sinh (c_t \sqrt{xy})
$$
where $c_{y,t}, c'_{y,t}$ are constants that depend only on $y$ and $t$.
It follows that
$$
(\log P_tf)'(x) = - \frac{1}{2x}  - \frac{1}{e^t-1} + \frac{c_t \sqrt{y}}{2\sqrt{x}}  \coth (c_t \sqrt{xy})
$$
and
\begin{align*}
(\log P_tf)''(x) 
& = 
\frac{1}{2x^2}  - \frac{c_t \sqrt{y}}{4 x\sqrt{x}}  \coth (c_t \sqrt{xy})
+\frac{c_t^2 y}{4 x}(1-\coth (c_t \sqrt{xy})^2) \\
& =
\frac{1}{4x^2}\left(2 + c_t^2xy- c_t \sqrt{xy} \coth (c_t \sqrt{xy}) - [c_t \sqrt{xy} \coth (c_t \sqrt{xy})]^2\right)\\
& =
\frac{1}{4x^2}\left(2 + z^2 - z \coth (z) - z^2 \coth (z)^2\right) 
=
\frac{1}{4x^2}\left(2 - \frac{z^2}{(\sinh z)^2} - z \coth (z) \right) 
\end{align*}
where in the third line we set $z=c_t \sqrt{xy}$. Since $z \coth z \to \infty$ as $z$ tends to infinity, and since $z/\sinh z \leq 1$ (for $z >0$), the latter shows that $(\log P_tf)"(x)$ cannot be bounded below by a constant independent on $f$ and $x$. Hence, log-semi-convexity (Problem (1) of the introduction) does not hold.

\subsection{Deviation bounds for log-semi-convex functions}

{In this section, we investigate deviation bounds for log-semi-convex functions. We prove that, due to the weak tail of the measures 
$\nu_\alpha$, the log-semi-convexity property does not help to get a better bound than Markov's inequality.
More precisely, setting $\mathcal{F}_{\beta,\alpha}:=\{f \geq 0: (\log f)"\geq  - \beta, \int fd\nu_\alpha=1\}$, $\beta \in \mathbb{R}$,
we have the following proposition.}

\begin{prop}
Let $\alpha >0$. Then, for all $\beta >0$, 
$$
\limsup_{t \to \infty} t \sup_{f \in \mathcal{F}_{\beta,\alpha}} \nu_\alpha(\{x : f(x) \geq t \}) > 0 .
$$
\end{prop}

\begin{proof}
We proceed as in the proof of Proposition \ref{prop:ce-mmi}. Fix $\beta >0$ and, for $a >0$, define
$f_a(x)=\exp\{-\frac{\beta}{2}(x-a)^2 + Z(a)\}$ where 
$Z(a):=-\log \int \exp\{-\frac{\beta}{2}(x-a)^2\}d\nu_\alpha(x)$ is devised so that $\int f_ad\nu_\alpha = 1$.
It is easy to prove (we omit details) that
$Z(a) \leq c \varphi_\alpha(a)$ for some positive constant $c$ that depends on $\beta$ and $\alpha$.
Hence,
\begin{align*}
\nu_\alpha(f_a \geq t ) 
& \geq 
\nu_\alpha \left( (x-a)^2 \leq \frac{2}{\beta}( \log t + \log \varphi_\alpha(a) + \log c) \right) .
\end{align*}
Now choose $a$ so that $\frac{2}{\beta}( \log t + \log \varphi_\alpha(a) + \log c) = 1$. Note that $a$ and $t$ 
must jointly go to infinity in order for the latter constraint to be satisfied, since $\varphi_\alpha(a) \to 0$ as $a \to \infty$. We infer that
$$
\nu_\alpha(f_a \geq t ) \geq \nu_\alpha( x \in [a-1,a+1] ) 
= 
\int_{a-1}^{a+1} \varphi_\alpha(x) dx 
\geq 
c' \varphi_\alpha(a)
$$
where the last inequality holds for $a$ (or equivalently $t$) large enough and follows after some approximation and algebra left to the reader (here $c'$ is a constant that depends only $\alpha$). But $a$ has been chosen so that $\varphi_\alpha(a)=\frac{c''}{t}$ for some constant $c''>0$ depending only on $\alpha$ and $\beta$. Hence,
$t \nu_\alpha(f_a \geq t ) \geq c''$ which proves the proposition.
\end{proof}

\begin{rem}
In \cite{GMRS17}, deviation bounds  for log-convex densities ($\beta=0$) under the exponential measure ($\alpha=1$) were deduced from the Gaussian case using a simple push-forward argument. The same argument could be easily used to get deviation bounds for log-convex functions for other Gamma distributions. 

As already mentioned, the above result ($\beta>0$) is due to weak tail of 
$\nu_\alpha$. Indeed, for such measures, we have 
$\int_x^\infty d\nu_\alpha \sim_{\infty} \varphi_\alpha(x)$, while for example for the standard Gaussian measure,
$\int_x^\infty e^{-t^2/2}dt \sim_{\infty} e^{-x^2/2}/x$, \textit{i.e.}\ there is a gain of a factor $1/x$ with respect to the Gaussian density in this case.
\end{rem}

\subsection{The Talagrand regularization effect}
In this final section, we prove Talagrand's regularization effect for Laguerre's semi-groups, by means of the strategy of the uniform bound on $P_t$. 
Such a regularization makes sense also in this setting since the Laguerre 
semi-groups enjoy an hypercontractive property \cite{Kor87, GLLNU05}.

\begin{thm}
Let $\alpha >0$ and denote by $(P_s^\alpha)_{s \geq 0}$ the Laguerre semi-group reversible with respect to $\nu_\alpha$. 
Then, for any $s>0$, there exists a constant $c$ (that depends only on $s$ and $\alpha$) such that for all { non-negative} 
real functions $f$ in $\mathbb{L}^1((0,\infty),\nu_\alpha)$ with $\int f d\nu_\alpha =1$, 
$$
\nu_\alpha(\{P_s^\alpha f \geq t \}) \leq  \frac{c}{t \sqrt{\log t}}, \qquad t >1 .
$$
\end{thm}

\begin{proof}
Fix $s>0$ and $t>1$.
We will use the strategy of the uniform bound on $P_t$. Namely, we first observe that
$$
\sup_{f \geq 0, \int f d\nu_\alpha=1} \nu_\alpha(\{P_s^\alpha f \geq t \})
\leq \nu_\alpha(\{ \sup_{f \geq 0, \int f d\nu_\alpha=1} P_s f \geq t \}) .
$$
Then, it is easy to see that, thanks to the kernel representation,
$$
 \sup_{f \geq 0, \int f d\nu_\alpha=1} P_s^\alpha f (x) = \sup_{y>0} G_s^\alpha(x,y), \qquad x >0 .
$$
Therefore we are left with an estimate on $G_s(x,y)$ (we look for an upper bound).
The following asymptotics are know \cite{AS64:book} to hold
$I_\beta(x) \sim_{\infty} \frac{e^x}{\sqrt{2 \pi x}}$ and $I_\beta(x) \sim_{0} \frac{\left(x/2\right)^\beta}{\Gamma(\beta+1)}$.
Up to a constant $c$ that depends on $\alpha$, we can safely assert that $I_{\alpha -1}(u) \leq ce^u/\sqrt{ u}$, for $u \geq 1$ and $I_{\alpha -1}(u) \leq c u^{\alpha-1}$ for $u \leq 1$.
In particular,
\begin{align*}
\sup_{y >0} G_s^\alpha(x,y) 
& \leq 
c' x^\frac{1-\alpha}{2} e^{-\frac{x}{e^s-1}} 
\max \left( x^\frac{\alpha-1}{2}  \sup_{0 < y \leq y_x} e^{-\frac{y}{e^s-1}} ;
x^{-\frac{1}{4}} \sup_{y >y_x} 
y^\frac{1-2\alpha}{4} e^{-\frac{y}{e^s-1} + \frac{2\sqrt{xye^s}}{e^s-1}} \right) \\
& =
c' x^\frac{1-\alpha}{2} e^{-\frac{x}{e^s-1}} 
\max \Big( x^\frac{\alpha-1}{2}  ; \\
& \qquad x^{-\frac{1}{4}} \exp\left\{ \frac{1}{e^s-1} \sup_{y >y_x} \frac{(1-2\alpha)(e^s-1)}{4}\log y - y + 2\sqrt{xye^s} \right\} \Big) 
\end{align*}
for some constant $c'$ that depends on $s$ and $\alpha$ and where $y_x$ is such that $\frac{2\sqrt{xy_xe^s}}{e^s-1}=1$,
\textit{i.e.}\ $y_x=\frac{(e^s-1)^2}{4xe^s}$. Hence, we need to bound from above
$$
\sup_{y>y_x} \frac{(1-2\alpha)(e^s-1)}{4} \log y - y + 2\sqrt{xye^s} 
=
\sup_{z>(e^s-1)/b} a \log z + bz - z^2 
$$
where we set $a=\frac{(1-2\alpha)(e^s-1)}{2}$ and $b=2\sqrt{xe^s}$ (and used the change of variable $z=\sqrt{y}$, together with the fact that
$\sqrt{y_x}=(e^s-1)/(2\sqrt{xe^s})=(e^s-1)/b$).
Denote by $H(y):=a \log z + bz - z^2$.
$c$ that depends only $s$ and $\alpha$. Hence, in this case, 
It is a tedious but easy exercise to prove that there exists a constant $c>0$ than depends only on $s$ and $\alpha$, and $x_o>0$ such that
$\sup_{z > \sqrt{y_x}} H(z) \leq c + \frac{a}{2}\log x + xe^s$ for $x \geq x_o$ and $\sup_{z > \sqrt{y_x}} H(z) \leq -\frac{1}{cx}$ for $x \leq x_o$. Hence, after some algebra
\begin{align*}
\sup_{y >0} G_s^\alpha(x,y) 
 \leq 
c' \left(1+x^\frac{1-2\alpha}{2} e^{x} \right)
\end{align*}
for some  constant $c'$ that depends only $s$ and $\alpha$. Denote by $F(x):=x + \frac{1-2\alpha}{2}\log x$, $x >0$ and observe that
$F$ increasing for $x$ large enough with inverse function we denote by $F^{-1}$ is also increasing. It is easy to see that $x \geq F^{-1}(x) \geq x - \frac{1-2\alpha}{2}\log x$ (for $x$ large enough). Therefore, for $t$ 
large enough
\begin{align*}
\sup_{f \geq 0, \int f d\nu_\alpha=1} \nu_\alpha(\{P_s^\alpha f \geq t \})
& \leq
\nu_\alpha \left(\left\{x : F(x) \geq \log \left( \frac{t}{c'}-1 \right) \right\} \right) \\
& \leq 
\nu_\alpha \left(\left\{x : x \geq F^{-1}\left(\log \left( \frac{t}{c'}-1 
\right)\right) \right\} \right) \\
& =
\frac{1}{\Gamma(\alpha)}\int_{F^{-1}\left(\log \left( \frac{t}{c'}-1 \right)\right)}^\infty x^{\alpha-1}e^{-x} dx \\
& \leq 
\kappa F^{-1}\left(\log \left( \frac{t}{c'}-1 \right)\right)^{\alpha-1}e^{-F^{-1}\left(\log \left( \frac{t}{c'}-1 \right)\right)} \\
& \leq
\kappa' (\log t)^{\alpha-1} e^{-\log \left( t\right) + \frac{1-2\alpha}{2} \log \log \left(t \right)}
=\kappa' \frac{1}{t \sqrt{\log t}}
\end{align*}
where we used that $\int_u^\infty x^{\alpha-1}e^{-x}dx \leq \kappa u^{\alpha-1}e^{-u}$ for $u$ large enough and $\kappa, \kappa'$ 
are constants that depends only on $\alpha$ and $s$. For $t$ close to $1$, the result follows from Markov's inequality (at the price of a possible 
bigger constant $\kappa'$).
\end{proof}

\medskip
\noindent{\bf Acknowledgment.} We thank Zhen-Qing Chen for enlightening discussions on the topic of this paper and an anonymous referee for his/her constructive comments.




\end{document}